\documentclass[a4paper,11pt]{article}

\usepackage{amsmath}
\usepackage{amsthm}
\usepackage{amssymb}
\usepackage{amsbsy}
\usepackage{mathrsfs}

\newtheorem{theorem}{Theorem}[section]
\newtheorem{lemma}[theorem]{Lemma}

\newtheorem{corollary}[theorem]{Corollary}

\newtheorem{claim}{Claim}

\newtheorem*{thm-ring-hyperhole}{Theorem~\ref{thm-ring-hyperhole}}

\begin{document}

\title{Coloring rings}

\author{Fr\'ed\'eric Maffray\thanks{CNRS, Laboratoire G-SCOP, Universit\'e Grenoble-Alpes, Grenoble, France.} \and Irena Penev\thanks{Computer Science Institute of Charles University (I\'{U}UK), Prague, Czech Republic. Email: \texttt{ipenev@iuuk.mff.cuni.cz}. Part of this research was conduced while the author was at the University of Leeds, Leeds, UK. Partially supported by project \texttt{17-04611S (Ramsey-like aspects of graph coloring)} of the Czech Science Foundation, by Charles University project \texttt{UNCE/SCI/004}, and by \texttt{EPSRC} grant \texttt{EP/N0196660/1}.} \and Kristina Vu\v{s}kovi\'c\thanks{School of Computing, University of Leeds, Leeds LS2 9JT, UK. Email: \texttt{k.vuskovic@leeds.ac.uk}. Partially supported by \texttt{EPSRC} grant \texttt{EP/N0196660/1} and by Serbian Ministry of Education and Science projects \texttt{174033} and \texttt{III44006}.}}

\date{\today}

\maketitle

\begin{abstract}
A {\em ring} is a graph $R$ whose vertex set can be partitioned into $k \geq 4$ nonempty sets, $X_1, \dots, X_k$, such that for all $i \in \{1,\dots,k\}$, the set $X_i$ can be ordered as $X_i = \{u_i^1, \dots, u_i^{|X_i|}\}$ so that $X_i \subseteq N_R[u_i^{|X_i|}] \subseteq \dots \subseteq N_R[u_i^1] = X_{i-1} \cup X_i \cup X_{i+1}$. A {\em hyperhole} is a ring $R$ such that for all $i \in \{1,\dots,k\}$, $X_i$ is complete to $X_{i-1}\cup X_{i+1}$. In this paper, we prove that the chromatic number of a ring $R$ is equal to the maximum chromatic number of a hyperhole in $R$. Using this result, we give a polynomial-time coloring algorithm for rings. 

Rings formed one of the basic classes in a decomposition theorem for a class of graphs studied by Boncompagni, Penev, and Vu\v{s}kovi\'c in [{\em Journal of Graph Theory} 91 (2019), 192--246]. Using our coloring algorithm for rings, we show that graphs in this larger class can also be colored in polynomial time. Furthermore, we find the optimal $\chi$-bounding function for this larger class of graphs, and we also verify Hadwiger's conjecture for it. 

\bigskip 

\end{abstract}

\section{Introduction}

All graphs in this paper are finite, simple, and nonnull. As usual, the vertex set and edge set of a graph $G$ are denoted by $V(G)$ and $E(G)$, respectively; for a vertex $v$ of $G$, $N_G(v)$ is the set of neighbors of $v$ in $G$, and $N_G[v] = N_G(v) \cup \{v\}$. 

A {\em ring} is a graph $R$ whose vertex set can be partitioned into $k \geq 4$ nonempty sets $X_1, \dots, X_k$ (whenever convenient, we consider indices of the $X_i$'s to be modulo $k$), such that for all $i \in \{1,\dots,k\}$ the set $X_i$ can be ordered as $X_i = \{u_i^1, \dots, u_i^{|X_i|}\}$ so that 
\begin{displaymath} 
X_i \subseteq N_R[u_i^{|X_i|}] \subseteq \dots \subseteq N_R[u_i^1] = X_{i-1} \cup X_i \cup X_{i+1}.
\end{displaymath} 
(Note that this implies that $X_1,\dots,X_k$ are all cliques\footnote{A {\em clique} is a set of pairwise adjacent vertices.} of $R$, and that $u_1^1,u_2^1,\dots,u_k^1,u_1^1$ is a hole\footnote{A {\em hole} is an induced cycle of length at least four.} of length $k$ in $R$.) Under such circumstances, we also say that the ring $R$ is of {\em length} $k$, or that $R$ is a {\em $k$-ring}; furthermore, $(X_1, \dots, X_k)$ is called a {\em ring partition} of $R$. A ring is {\em even} or {\em odd} depending on the parity of its length. Rings played an important role in~\cite{VIK}: they formed a ``basic class'' in the decomposition theorems for a couple of graph classes defined by excluding certain ``Truemper configurations'' as induced subgraphs (more on this in subsection~\ref{subsec:Truemper}). In that paper, the complexity of the optimal vertex coloring problem for rings was left as an open problem.\footnote{In fact, only odd rings are difficult in this regard; even rings are readily colorable in polynomial time (see Lemma~\ref{lemma-even-ring-col-alg}).} In the present paper, we give a polynomial-time coloring algorithm for rings (see Theorems~\ref{thm-ring-col-alg} and~\ref{thm-ring-chi-alg}). 

It can easily be shown that every ring is a circular-arc graph. Furthermore, rings have unbounded clique-width. To see this, let $k \geq 3$ be an integer, and let $R$ be a $(k+1)$-ring with ring partition $(X_1,\dots,X_k,X_{k+1})$ such that the cliques $X_i$ are all of size $k+1$, with vertices labeled $0,1,\dots,k$, and furthermore, assume that vertices labeled $p$ and $q$ from consecutive cliques of the ring partition are adjacent if and only if $p+q \leq k$. Now, the graph obtained from $R$ by first deleting $X_{k+1}$, and then deleting all the vertices labeled 0, is precisely the permutation graph $H_k$ defined in~\cite{gr}, and the clique-width of $H_k$ is at least $k$ (see Lemma~5.4 from~\cite{gr}). 

Given graphs $H$ and $G$, we say that $G$ {\em contains} $H$ if $G$ contains an induced subgraph isomorphic to $H$; if $G$ does not contain $H$, then $G$ is {\em $H$-free}. For a family $\mathcal{H}$ of graphs, we say that a graph $G$ is {\em $\mathcal{H}$-free} if $G$ is $H$-free for all $H \in \mathcal{H}$. 

Given a graph $G$, a {\em clique} of $G$ is a (possibly empty) set of pairwise adjacent vertices of $G$, and a {\em stable set} of $G$ is a (possibly empty) set of pairwise nonadjacent vertices of $G$. The {\em clique number} of $G$, denoted by $\omega(G)$, is the maximum size of a clique of $G$, and the {\em stability number} of $G$, denoted by $\alpha(G)$, is the maximum size of a stable set of $G$. A {\em proper coloring} of $G$ is an assignment of colors to the vertices of $G$ in such a way that no two adjacent vertices receive the same color. For a positive integer $r$, $G$ is said to be {\em $r$-colorable} if there is a proper coloring of $G$ that uses at most $r$ colors. The {\em chromatic number} of $G$, denoted by $\chi(G)$, is the minimum number of colors needed to properly color $G$. An {\em optimal coloring} of $G$ is a proper coloring of $G$ that uses only $\chi(G)$ colors. 

Given a graph $G$, a vertex $v \in V(G)$, and a set $S \subseteq V(G) \setminus \{v\}$, we say that $v$ is {\em complete} (resp.\ {\em anticomplete}) to $S$ in $G$ provided that $v$ is adjacent (resp.\ nonadjacent) to every vertex of $S$; given disjoint sets $X,Y \subseteq V(G)$, we say that $X$ is {\em complete} (resp.\ {\em anticomplete}) to $Y$ in $G$ provided that every vertex in $X$ is complete (resp.\ anticomplete) to $Y$ in $G$. 

A {\em hole} is a chordless cycle on at least four vertices; the {\em length} of a hole is the number of its vertices, and a hole is \emph{even} or \emph{odd} according to the parity of its length. When we say ``$H$ is a hole in $G$,'' we mean that $H$ is a hole that is an induced subgraph of $G$. 

A \emph{hyperhole} is any graph $H$ whose vertex set can be partitioned into $k \geq 4$ nonempty cliques $X_1, \dots, X_k$ (whenever convenient, we consider indices of the $X_i$'s to be modulo $k$) such that for all $i \in \{1,\dots,k\}$, $X_i$ is complete to $X_{i-1}\cup X_{i+1}$ and anticomplete to $V(H) \setminus (X_{i-1} \cup X_i \cup X_{i+1})$; under such circumstances, we also say that $H$ is a hyperhole of {\em length} $k$, or that $H$ is a {\em $k$-hyperhole}. A hyperhole is {\em even} or {\em odd} according to the parity of its length. Note that every hole is a hyperhole, and every hyperhole is a ring. When we say ``$H$ is a hyperhole in $G$,'' we mean that $H$ is a hyperhole that is an induced subgraph of $G$. 

Hyperholes can be colored in linear time~\cite{hyperhole}. Furthermore, the following lemma gives a formula for the chromatic number of a hyperhole. 

\begin{lemma} \cite{hyperhole} \label{lemma-hyperhole-chi-formula} Let $H$ be a hyperhole. Then $\chi(H) = \max\Big\{\omega(H),\Big\lceil \frac{|V(H)|}{\alpha(H)} \Big\rceil\Big\}$. 
\end{lemma} 

The main result of the present paper is the following theorem. 

\begin{theorem} \label{thm-ring-hyperhole} Let $k \geq 4$ be an integer, and let $R$ be a $k$-ring. Then $\chi(R) = \max\{\chi(H) \mid \text{$H$ is a $k$-hyperhole in $R$}\}$.
\end{theorem} 

It was shown in~\cite{VIK} that all holes of a $k$-ring ($k \geq 4$) are of length $k$;\footnote{In the present paper, this result is stated as Lemma~\ref{lemma-ring-chordal}(b).} consequently, all hyperholes in a $k$-ring are of length $k$. Thus, Theorem~\ref{thm-ring-hyperhole} in fact establishes that the chromatic number of a ring is equal to the maximum chromatic number of a hyperhole in the ring. 

It is easy to see that the stability number of any $k$-hyperhole ($k \geq 4$) is $\lfloor k/2 \rfloor$. Thus, the following is an immediate corollary of Lemma~\ref{lemma-hyperhole-chi-formula} and Theorem~\ref{thm-ring-hyperhole}. 

\begin{corollary} \label{cor-ring-chi-formula} Let $k \geq 4$ be an integer, and let $R$ be a $k$-ring. Then $\chi(R) = \max\Big(\{\omega(R)\} \cup \Big\{\Big\lceil \frac{|V(H)|}{\lfloor k/2 \rfloor} \Big\rceil \mid \text{$H$ is a $k$-hyperhole in $R$}\Big\}\Big)$. 
\end{corollary} 

Using Theorem~\ref{thm-ring-hyperhole},\footnote{More precisely, we use Lemma~\ref{lemma-reduce-chi}, which is a corollary of Theorem~\ref{thm-ring-hyperhole} and Lemma~\ref{lemma-main-technical}. Lemma~\ref{lemma-main-technical}, in turn, is the main part of the proof of Theorem~\ref{thm-ring-hyperhole}.} we construct an $O(n^6)$ algorithm that computes an optimal coloring of an input ring (see Theorem~\ref{thm-ring-col-alg}). Furthermore, using Corollary~\ref{cor-ring-chi-formula}, we also give an $O(n^3)$ time algorithm that computes the chromatic number of a ring without actually finding an optimal coloring of that ring (see Theorem~\ref{thm-ring-chi-alg}).

\subsection{Terminology, notation, and paper outline} \label{subsec:Truemper} 

For a function $f:A \rightarrow B$ and a set $A' \subseteq A$, we denote by $f \upharpoonright A'$ the restriction of $f$ to $A'$. 

The complement of a graph $G$ is denoted by $\overline{G}$. For a nonempty set $X \subseteq V(G)$, we denote by $G[X]$ the subgraph of $G$ induced by $X$; for vertices $x_1,\dots,x_t \in V(G)$, we often write $G[x_1,\dots,x_t]$ instead of $G[\{x_1,\dots,x_t\}]$. For a set $S \subsetneqq V(G)$, we denote by $G \setminus S$ the subgraph of $G$ obtained by deleting $S$, i.e.\ $G \setminus S = G[V(G) \setminus S]$; if $G$ has at least two vertices and $v \in V(G)$, then we often write $G \setminus v$ instead of $G \setminus \{v\}$.\footnote{Since our graphs are nonnull, if $G$ has just one vertex, say $v$, then $G \setminus v$ is undefined.} 

A class of graphs is {\em hereditary} if it is closed under isomorphism and induced subgraphs. More precisely, a class $\mathcal{G}$ of graphs is {\em hereditary} if for every graph $G \in \mathcal{G}$, the class $\mathcal{G}$ contains all isomorphic copies of induced subgraphs of $G$. 

A {\em theta} is any subdivision of the complete bipartite graph $K_{2,3}$; in particular, $K_{2,3}$ is a theta. A {\em pyramid} is any subdivision of the complete graph $K_4$ in which one triangle remains unsubdivided, and of the remaining three edges, at least two edges are subdivided at least once. A {\em prism} is any subdivision of $\overline{C_6}$ in which the two triangles remain unsubdivided; in particular, $\overline{C_6}$ is a prism. A {\em three-path-configuration} (or {\em 3PC} for short) is any theta, pyramid, or prism. 

A {\em wheel} is a graph that consists of a hole and an additional vertex that has at least three neighbors in the hole. If this additional vertex is adjacent to all vertices of the hole, then the wheel is said to be a {\em universal wheel}; if the additional vertex is adjacent to three consecutive vertices of the hole, and to no other vertex of the hole, then the wheel is said to be a {\em twin wheel}. A {\em proper wheel} is a wheel that is neither a universal wheel nor a twin wheel. 

A {\em Truemper configuration} is any 3PC or wheel (for a survey, see~\cite{Truemper-survey}). Note that every Truemper configuration contains a hole. Note, furthermore, that every prism or theta contains an even hole, and every pyramid contains an odd hole. Thus, even-hole-free graphs contain no prisms and no thetas, and odd-hole-free graphs contain no pyramids. 

$\mathcal{G}_{\text{T}}$ is the class of all (3PC, proper wheel, universal wheel)-free graphs; thus, the only Truemper configurations that a graph in $\mathcal{G}_{\text{T}}$ may contain are the twin wheels. Clearly, the class $\mathcal{G}_{\text{T}}$ is hereditary. A decomposition theorem for $\mathcal{G}_{\text{T}}$ (where rings form one of the ``basic classes'') was obtained in~\cite{VIK},\footnote{In the present paper, this decomposition theorem is stated as Theorem~\ref{thm-GT-decomp}.} as were polynomial-time algorithms that solve the recognition, maximum weight clique, and maximum weight stable set problems for the class $\mathcal{G}_{\text{T}}$. The complexity of the optimal coloring problem for $\mathcal{G}_{\text{T}}$ was left open in~\cite{VIK}, and the main obstacle in this context were rings. In the present paper, we show that graphs in $\mathcal{G}_{\text{T}}$ can be colored in polynomial time (see Theorems~\ref{thm-GT-col-alg} and~\ref{thm-GT-chi-alg}). 

A {\em simplicial vertex} is a vertex whose neighborhood is a (possibly empty) clique. For an integer $k \geq 4$, let $\mathcal{R}_k$ be the class of all graphs $G$ that have the property that every induced subgraph of $G$ either is a $k$-ring or has a simplicial vertex; clearly, $\mathcal{R}_k$ is hereditary, and furthermore (by Lemma~\ref{lemma-RRk-hered}) it contains all $k$-rings. We remark that graphs in $\mathcal{R}_k$ are precisely the chordal graphs,\footnote{A graph is {\em chordal} if it contains no holes.} and the graphs that can be obtained from a $k$-ring by (possibly) repeatedly adding simplicial vertices (see Lemma~\ref{lemma-RRk-description}). Further, for all integers $k \geq 4$, we set $\mathcal{R}_{\geq k} = \bigcup\limits_{i=k}^{\infty} \mathcal{R}_i$; clearly, $\mathcal{R}_{\geq k}$ is hereditary, and furthermore (by Lemma~\ref{lemma-RRk-hered}) it contains all rings of length at least $k$. In particular, the class $\mathcal{R}_{\geq 4}$ is hereditary and contains all rings. We show that graphs in $\mathcal{R}_{\geq 4}$ can be colored in polynomial time (see Theorems~\ref{thm-ring-col-alg} and~\ref{thm-ring-chi-alg}). 

A {\em clique-cutset} of a graph $G$ is a (possibly empty) clique $C \subsetneqq V(G)$ of $G$ such that $G \setminus C$ is disconnected. A {\em clique-cut-partition} of a graph $G$ is a partition $(A,B,C)$ of $V(G)$ such that $A$ and $B$ are nonempty and anticomplete to each other, and $C$ is a (possibly empty) clique. Clearly, a graph admits a clique-cutset if and only if it admits a clique-cut-partition. 

A graph is {\em perfect} if all its induced subgraphs $H$ satisfy $\chi(H) = \omega(H)$. The Strong Perfect Graph Theorem~\cite{SPGT} states that a graph $G$ is perfect if and only if neither $G$ nor $\overline{G}$ contains an odd hole. 

$\mathbb{N}$ is the set of all positive integers. A hereditary class $\mathcal{G}$ is {\em $\chi$-bounded} if there exists a function $f:\mathbb{N} \rightarrow \mathbb{N}$ (called a {\em $\chi$-bounding function} for $\mathcal{G}$) such that all graphs $G \in \mathcal{G}$ satisfy $\chi(G) \leq f(\omega(G))$. For a hereditary $\chi$-bounded class $\mathcal{G}$ that contains all complete graphs (equivalently: that contains graphs of arbitrarily large clique number), we say that a $\chi$-bounding function $f:\mathbb{N} \rightarrow \mathbb{N}$ for $\mathcal{G}$ is {\em optimal} if for all $n \in \mathbb{N}$, there exists a graph $G \in \mathcal{G}$ such that $\omega(G) = n$ and $\chi(G) = f(n)$. It was shown in~\cite{VIK} that $\mathcal{G}_{\text{T}}$ is $\chi$-bounded by a linear function; more precisely, it was shown that every graph $G \in \mathcal{G}_{\text{T}}$ satisfies $\chi(G) \leq \Big\lfloor \frac{3}{2}\omega(G) \Big\rfloor$.\footnote{See Theorem~7.6 from~\cite{VIK}.} In the present paper, we improve this $\chi$-bounding function, and in fact, we find the optimal $\chi$-bounding function for the class $\mathcal{G}_{\text{T}}$ (see Theorem~\ref{thm-GT-chi}). 

Finally, we consider Hadwiger's conjecture. Let $H$ be an $n$-vertex graph with vertex set $V(H) = \{v_1,\dots,v_n\}$. We say that a graph $G$ {\em contains $H$ as a minor} if there exist pairwise disjoint, nonempty subsets $S_1,\dots,S_n \subseteq V(G)$ (called {\em branch sets}) such that $G[S_1],\dots,G[S_n]$ are all connected, and such that for all distinct $i,j \in \{1,\dots,n\}$ with $v_iv_j \in E(H)$, there is at least one edge between $S_i$ and $S_j$ in $G$. As usual, the complete graph on $k$ vertices is denoted by $K_k$. Hadwiger's conjecture states that every graph $G$ contains $K_{\chi(G)}$ as a minor. Using Theorem~\ref{thm-ring-hyperhole}, we prove that rings satisfy Hadwiger's conjecture (see Lemma~\ref{lemma-Hadwiger-ring}), and as a corollary, we obtain that graphs in $\mathcal{G}_{\text{T}}$ also satisfy Hadwiger's conjecture (see Theorem~\ref{thm-Hadwiger-GT}). 

A {\em hyperantihole} is a graph $A$ whose vertex set can be partitioned into nonempty cliques $X_1,\dots,X_k$ ($k \geq 4$)\footnote{Whenever convenient, we consider indices of the $X_i$'s to be modulo $k$.} such that for all $i \in \{1,\dots,k\}$, $X_i$ is complete to $V(A) \setminus (X_{i-1} \cup X_i \cup X_{i+1})$ and anticomplete to $X_{i-1} \cup X_{i+1}$.\footnote{Note that the complement of a hyperantihole need not be a hyperhole.} Under these circumstances, we also say that the hyperantihole $A$ is of {\em length} $k$, and that $A$ is a {\em $k$-hyperantihole}. A hyperantihole is {\em odd} or {\em even} depending on the parity of its length. 

The remainder of this paper is organized as follows. In section~\ref{sec:prelim}, we state a few results from the literature that we need in the remainder of the paper; in section~\ref{sec:prelim}, we also prove a few easy lemmas about rings and their induced subgraphs, and about classes $\mathcal{R}_k$ and $\mathcal{R}_{\geq k}$ ($k \geq 4$). In section~\ref{sec:main}, we prove Theorem~\ref{thm-ring-hyperhole}, and we also give a polynomial-time coloring algorithm for even rings (see Lemma~\ref{lemma-even-ring-col-alg}). In section~\ref{sec:col}, we give an $O(n^6)$ time coloring algorithm for rings (see Theorem~\ref{thm-ring-col-alg}).\footnote{In fact, this is a coloring algorithm for graphs in $\mathcal{R}_{\geq 4}$. By Lemma~\ref{lemma-RRk-hered}, $\mathcal{R}_{\geq 4}$ contains all rings.} Even rings are easy to color (see Lemma~\ref{lemma-even-ring-col-alg}); our coloring algorithm for odd rings relies on ideas from the proof of Theorem~\ref{thm-ring-hyperhole}. Using our coloring algorithm for rings, as well as various results from the literature, we also construct an $O(n^7)$ time coloring algorithm for graphs in $\mathcal{G}_{\text{T}}$ (see Theorem~\ref{thm-GT-col-alg}). In section~\ref{sec:chi}, we construct an $O(n^3)$ time algorithm that computes the chromatic number of a ring (see Theorem~\ref{thm-ring-chi-alg}),\footnote{In fact, our algorithm computes the chromatic number of graphs in $\mathcal{R}_{\geq 4}$.} and more generally, we construct an $O(n^5)$ time algorithm that computes the chromatic number of graphs in $\mathcal{G}_{\text{T}}$ (see Theorem~\ref{thm-GT-chi-alg}).\footnote{The difference between the algorithms from Theorems~\ref{thm-ring-col-alg} and~\ref{thm-GT-col-alg} on the one hand, and the algorithms from Theorems~\ref{thm-ring-chi-alg} and~\ref{thm-GT-chi-alg} on the other, is that the former compute an optimal coloring of the input graph from the relevant class, whereas the latter only compute the chromatic number (but are significantly faster than the former).} In section~\ref{sec:chibound}, we obtain the optimal $\chi$-bounding function for the class $\mathcal{G}_{\text{T}}$ (see Theorem~\ref{thm-GT-chi}). Furthermore, in section~\ref{sec:chibound}, for each odd integer $k \geq 5$, we obtain the optimal bound for the chromatic number in terms of the clique number for $k$-hyperholes and $k$-hyperantiholes.\footnote{We only defined $\chi$-boundedness for hereditary classes, and so, technically, these are not ``$\chi$-bounding functions'' for the classes of $k$-hyperholes and $k$-hyperantiholes. They are, however, the optimal $\chi$-bounding functions for the closures of these classes under induced subgraphs. See section~\ref{sec:chibound} for the details.} Finally, in section~\ref{sec:Hadwiger}, we prove Hadwiger's conjecture for the class $\mathcal{G}_{\text{T}}$ (see Theorem~\ref{thm-Hadwiger-GT}).

\section{A few preliminary lemmas} \label{sec:prelim}

In this section, we state a few results from the literature, which we use later in the paper. We also prove a few easy results about rings and their induced subgraphs, and about classes $\mathcal{R}_k$ and $\mathcal{R}_{\geq k}$ ($k \geq 4$).  

Given a graph $G$ and distinct vertices $u,v \in V(G)$, we say that $u$ {\em dominates} $v$ in $G$, and that $v$ is {\em dominated} by $u$ in $G$, whenever $N_G[v] \subseteq N_G[u]$. The following lemma was stated without proof in~\cite{VIK} (see Lemma~1.4 from~\cite{VIK}); it readily follows from the definition of a ring, as the reader can check. 

\begin{lemma} \cite{VIK} \label{lemma-ring-char} Let $G$ be a graph, and let $(X_1,\dots,X_k)$, with $k \geq 4$, be a partition of $V(G)$. Then $G$ is a $k$-ring with ring partition $(X_1,\dots,X_k)$ if and only if all the following hold:\footnote{As usual, indices of the $X_i$'s are understood to be modulo $k$.} 
\begin{itemize} 
\item[(a)] $X_1,\dots,X_k$ are cliques; 
\item[(b)] for all $i \in \{1,\dots,k\}$, $X_i$ is anticomplete to $V(G) \setminus (X_{i-1} \cup X_i \cup X_{i+1})$; 
\item[(c)] for all $i \in \{1,\dots,k\}$, some vertex of $X_i$ is complete to $X_{i-1} \cup X_{i+1}$; 
\item[(d)] for all $i \in \{1,\dots,k\}$, and all distinct $y_i,y_i' \in X_i$, one of $y_i,y_i'$ dominates the other. 
\end{itemize} 
\end{lemma} 

Recall that a graph is {\em chordal} if it contains no holes. The following is Lemma~2.4(a)-(d) from~\cite{VIK}. 

\begin{lemma} \cite{VIK} \label{lemma-ring-chordal} Let $R$ be a $k$-ring ($k \geq 4$) with ring partition $(X_1,\dots,X_k)$. Then all the following hold: 
\begin{itemize} 
\item[(a)] every hole in $R$ intersects each of $X_1,\dots,X_k$ in exactly one vertex; 
\item[(b)] every hole in $R$ is of length $k$; 
\item[(c)] for all $i \in \{1,\dots,k\}$, $R \setminus X_i$ is chordal; 
\item[(d)] $R \in \mathcal{G}_{\text{T}}$. 
\end{itemize} 
\end{lemma} 

Note that Lemma~\ref{lemma-ring-chordal}(b) states that, for an integer $k \geq 4$, every hyperhole in a $k$-ring is of length $k$. On the other hand, Lemma~\ref{lemma-ring-chordal}(d) implies that $\mathcal{R}_{\geq 4} \subseteq \mathcal{G}_{\text{T}}$,\footnote{Indeed, suppose that $G \in \mathcal{R}_{\geq 4}$, and assume inductively that all graphs in $\mathcal{R}_{\geq 4}$ on fewer than $|V(G)|$ vertices belong to $\mathcal{G}_{\text{T}}$. If $G$ is a ring, then Lemma~\ref{lemma-ring-chordal}(d) guarantees that $G \in \mathcal{G}_{\text{T}}$. So suppose that $G$ is not a ring. Then by the definition of $\mathcal{R}_{\geq 4}$, $G$ has a simplicial vertex, call it $v$. Obviously, $K_1 \in \mathcal{G}_{\text{T}}$, and so we may assume that $|V(G)| \geq 2$. Note that no Truemper configuration contains a simplicial vertex, and so $v$ does not belong to any induced Truemper configuration in $G$. Since $\mathcal{G}_{\text{T}}$ was defined by forbidding certain Truemper configurations as induced subgraphs, we deduce that $G$ belongs to $\mathcal{G}_{\text{T}}$ if and only if $G \setminus v$ does. Now, since $\mathcal{R}_{\geq 4}$ is hereditary and contains $G$, we see that $G \setminus v$ belongs to $\mathcal{R}_{\geq 4}$. It then follows from the induction hypothesis that $G \setminus v$ belongs to $\mathcal{G}_{\text{T}}$, and we deduce that $G \in \mathcal{G}_{\text{T}}$.} but we will not need this in the remainder of the paper.

Rings can be recognized in polynomial time. More precisely, the following is Lemma~8.14 from~\cite{VIK}. (In all our algorithms, $n$ denotes the number of vertices and $m$ the number of edges of the input graph.) 

\begin{lemma} \cite{VIK} \label{lemma-detect-ring} There exists an algorithm with the following specifications: 
\begin{itemize} 
\item Input: A graph $G$; 
\item Output: Either the true statement that $G$ is a ring, together with the length and ring partition of the ring, or the true statement that $G$ is not a ring; 
\item Running time: $O(n^2)$. 
\end{itemize} 
\end{lemma} 

As an easy corollary of Lemma~\ref{lemma-detect-ring}, we can obtain Lemma~\ref{lemma-detect-ring-ord} (below). We remark that the proof (but not the statement) of Lemma~8.14 from~\cite{VIK} in fact gives precisely Lemma~\ref{lemma-detect-ring-ord}. For the sake of completeness, we give a full proof. 

\begin{lemma} \label{lemma-detect-ring-ord} There exists an algorithm with the following specifications: 
\begin{itemize} 
\item Input: A graph $G$; 
\item Output: Exactly one of the following: 
\begin{itemize} 
\item the true statement that $G$ is a ring, together with the length $k$ and a ring partition $(X_1,\dots,X_k)$ of the ring $G$, and for each $i \in \{1,\dots,k\}$, an ordering $X_i = \{u_i^1,\dots,u_i^{|X_i|}\}$ of $X_i$ such that $X_i \subseteq N_G[u_i^{|X_i|}] \subseteq \dots \subseteq N_G[u_i^1] = X_{i-1} \cup X_i \cup X_{i+1}$, 
\item the true statement that $G$ is not a ring; 
\end{itemize} 
\item Running time: $O(n^2)$. 
\end{itemize} 
\end{lemma} 
\begin{proof} 
We first run the algorithm from Lemma~\ref{lemma-detect-ring} with input $G$; this takes $O(n^2)$ time. If the algorithm returns the statement that $G$ is not a ring, then we return this statement as well, and we stop. So assume that the algorithm returned the statement that $G$ is a ring, together with the length $k$ and ring partition $(X_1,\dots,X_k)$ of the ring. We then find the degrees of all vertices of $G$, and for each $i \in \{1,\dots,k\}$, we order $X_i$ as $X_i = \{u_i^1,\dots,u_i^{|X_i|}\}$ so that ${\rm deg}_G(u_i^1) \geq \dots \geq {\rm deg}_G(u_i^{|X_i|})$; this takes $O(n^2)$ time. Since we already know that $(X_1,\dots,X_k)$ is a ring partition of $G$, it is easy to see that for all $i \in \{1,\dots,k\}$, we have that $X_i \subseteq N_G[u_i^{|X_i|}] \subseteq \dots \subseteq N_G[u_i^1] = X_{i-1} \cup X_i \cup X_{i+1}$. We now return the statement that $G$ is a ring of length $k$, the ring partition $(X_1,\dots,X_k)$ of $G$, and for each $i \in \{1,\dots,k\}$, the ordering $X_i = \{u_i^1,\dots,u_i^{|X_i|}\}$ of $X_i$, and we stop. Clearly, the algorithm is correct, and its running time is $O(n^2)$. 
\end{proof} 

We remind the reader that a {\em simplicial vertex} is a vertex whose neighborhood is a (possibly empty) clique. A {\em simplicial elimination ordering} of a graph $G$ is an ordering $v_1,\dots,v_n$ of the vertices of $G$ such that for all $i \in \{1,\dots,n\}$, $v_i$ is simplicial in the graph $G[v_i,v_{i+1},\dots,v_n]$. It is well known (and easy to show) that a graph is chordal if and only if it has a simplicial elimination ordering (see~\cite{FulkersonGross}); in particular, every chordal graph contains a simplicial vertex. We also note that there is an $O(n+m)$ time algorithm that either produces a simplicial elimination ordering of the input graph, or determines that the graph is not chordal~\cite{RTL76}. Recall that a graph is {\em perfect} if all its induced subgraphs $H$ satisfy $\chi(H) = \omega(H)$; it is well known (and easy to show) that chordal graphs are perfect~\cite{Berge-German, D61}. 

The following algorithm is a minor modification of the algorithm described in the introduction of~\cite{HHMMSimplicial}.\footnote{The algorithm from~\cite{HHMMSimplicial} produces a maximal sequence $v_1,\dots,v_t$ ($t \geq 0$) of pairwise distinct vertices of the input graph $G$ such that for all $i \in \{1,\dots,t\}$, $v_i$ is simplicial in either $G \setminus \{v_1,\dots,v_{i-1}\}$ or $\overline{G} \setminus \{v_1,\dots,v_{i-1}\}$. Thus, the algorithm from Lemma~\ref{lemma-simplicial-list-alg} is in fact obtained from the algorithm from~\cite{HHMMSimplicial} by omitting some steps. The running time of the two algorithms is the same. For the sake of completeness, we give all the details for the algorithm that we need (i.e.\ for the algorithm from Lemma~\ref{lemma-simplicial-list-alg}).} 

\begin{lemma} \label{lemma-simplicial-list-alg} There exists an algorithm with the following specifications: 
\begin{itemize} 
\item Input: A graph $G$; 
\item Output: A maximal sequence $v_1,\dots,v_t$ ($t \geq 0$) of pairwise distinct vertices of $G$ such that for all $i \in \{1,\dots,t\}$, $v_i$ is simplicial in the graph $G \setminus \{v_1,\dots,v_{i-1}\}$;\footnote{If $t = 0$, then the sequence $v_1,\dots,v_t$ is empty and $G$ has no simplicial vertices.} 
\item Running time: $O(n^3)$. 
\end{itemize} 
\end{lemma} 
\begin{proof}
\textbf{Step~0.} First, for all distinct $x,y \in V(G)$, we set 
\begin{displaymath} 
\begin{array}{ccc} 
\text{diff}(x,y) & = & \left\{\begin{array}{lll} 
|N_G[x] \setminus N_G[y]| & \text{if} & xy \in E(G) 
\\
\\
0 & \text{if} & xy \notin E(G) 
\end{array}\right. 
\end{array} 
\end{displaymath} 
Clearly, computing $\text{diff}(x,y)$ for all possible choices of distinct $x,y \in V(G)$ can be done in $O(n^3)$ time. We will update $\text{diff}(x,y)$ as the algorithm proceeds. Note that a vertex $x \in V(G)$ is simplicial in $G$ if and only if for all $y \in V(G) \setminus \{x\}$, we have that $\text{diff}(x,y) = 0$. Let $L$ be the empty list. We now go to Step~1. 

\textbf{Step~1.} We first check if there is a vertex $x \in V(G)$ such that for all $y \in V(G) \setminus \{x\}$, we have that $\text{diff}(x,y) = 0$; this can be done in $O(n^2)$ time. If we found no such vertex, then $G$ has no simplicial vertices; in this case, we return the list $L$ and stop. Suppose now that we found such a vertex $x$. First, we set $L := L,x$ (i.e.\ we update $L$ by adding $x$ to the end of $L$). If $x$ is the only vertex of $G$, then we return $L$ and stop. Suppose now that $G$ has at least two vertices. Then, for all distinct $x',y \in V(G) \setminus \{x\}$, we update $\text{diff}(x',y)$ as follows: if $x \in N_G[x'] \setminus N_G[y]$, then we set $\text{diff}(x',y) := \text{diff}(x',y)-1$, and otherwise, we do not change $\text{diff}(x',y)$; this update takes $O(n^2)$ time. Finally, we update $G$ by setting $G := G \setminus x$, and we go to Step~1 with input $G$, $L$, and $\text{diff}(x',y)$ for all distinct $x',y \in V(G)$. 

Clearly, the algorithm terminates and is correct. Step~0 takes $O(n^3)$ time. We make $O(n)$ calls to Step~1, and otherwise, the slowest step of Step~1 takes $O(n^2)$ time. Thus, the total running time of the algorithm is $O(n^3)$. 
\end{proof} 

Recall that chordal graphs are precisely those graphs that admit a simplicial elimination ordering~\cite{FulkersonGross}. So, the algorithm from Lemma~\ref{lemma-simplicial-list-alg} can be used to recognize chordal graphs in $O(n^3)$ time.\footnote{Indeed, suppose that, given an $n$-vertex input graph $G$, the algorithm from Lemma~\ref{lemma-simplicial-list-alg} returned the sequence $v_1,\dots,v_t$. If $t = n$ (i.e.\ $V(G) = \{v_1,\dots,v_t\}$), then $v_1,\dots,v_t$ is a simplicial elimination ordering of $G$, and therefore (by~\cite{FulkersonGross}) $G$ is chordal. Suppose now that $t < n$. Then the maximality of $v_1,\dots,v_t$ guarantees that $G \setminus \{v_1,\dots,v_t\}$ has no simplicial vertices. Then by~\cite{FulkersonGross}, $G \setminus \{v_1,\dots,v_t\}$ is not chordal, and consequently, $G$ is not chordal either.} 
 
Lemma~\ref{lemma-omega-ring} (below) follows immediately from Theorem~8.25 from~\cite{VIK}. 

\begin{lemma} \cite{VIK} \label{lemma-omega-ring} There exists an algorithm with the following specifications: 
\begin{itemize} 
\item Input: A graph $G$; 
\item Output: Either $\omega(G)$, or the true statement that $G \notin \mathcal{G}_{\text{T}}$; 
\item Running time: $O(n^3)$. 
\end{itemize} 
\end{lemma} 

By Lemma~\ref{lemma-ring-chordal}(d), rings belong to $\mathcal{G}_{\text{T}}$, and so Lemma~\ref{lemma-omega-ring} guarantees that the clique number of a ring can be computed in $O(n^3)$ time. 

\begin{lemma} \label{lemma-ring-simplicial} Let $k \geq 4$ be an integer. Then every induced subgraph of a $k$-ring either contains a simplicial vertex or is a $k$-ring. More precisely, let $R$ be a $k$-ring with ring partition $(X_1,\dots,X_k)$, and let $Y \subseteq V(R)$ be a nonempty set. Then either $R[Y]$  contains a simplicial vertex, or $R[Y]$ is a $k$-ring with ring partition $(X_1 \cap Y,\dots,X_k \cap Y)$. 
\end{lemma} 
\begin{proof} 
For all $i \in \{1,\dots,k\}$, we set $X_i = \{u_i^1, \dots, u_i^{|X_i|}\}$ so that $X_i \subseteq N_R[u_i^{|X_i|}] \subseteq \dots \subseteq N_R[u_i^1] = X_{i-1} \cup X_i \cup X_{i+1}$, as in the definition of a ring. For all $i \in \{1,\dots,k\}$, we set $Y_i = X_i \cap Y$. If at least one of $Y_1,\dots,Y_k$ is empty, then Lemma~\ref{lemma-ring-chordal}(c) implies that $R[Y]$ is chordal, and consequently (by~\cite{FulkersonGross}), $R[Y]$ contains a simplicial vertex. So from now on, we assume that $Y_1,\dots,Y_k$ are all nonempty. 

For all $i \in \{1,\dots,k\}$, let $j_i \in \{1,\dots,|X_i|\}$ be maximal with the property that $u_i^{j_i} \in Y_i$; then $u_i^{j_i}$ is dominated in $R[Y]$ by all other vertices in $Y_i$. If for some $i \in \{1,\dots,k\}$, $u_i^{j_i}$ is anticomplete to $Y_{i-1}$ or $Y_{i+1}$, then it is easy to see that $u_i^{j_i}$ is a simplicial vertex of $R[Y]$, and we are done; otherwise, Lemma~\ref{lemma-ring-char} implies that $R[Y]$ is a ring with ring partition $(Y_1,\dots,Y_k)$. 
\end{proof} 

\begin{lemma} \label{lemma-RRk-hered} For all integers $k \geq 4$, both the following hold: 
\begin{itemize} 
\item the class $\mathcal{R}_k$ is hereditary and contains all $k$-rings; 
\item the class $\mathcal{R}_{\geq k}$ is hereditary and contains all rings of length at least $k$. 
\end{itemize} 
In particular, the class $\mathcal{R}_{\geq 4}$ is hereditary and contains all rings. 
\end{lemma} 
\begin{proof} 
This follows immediately from Lemma~\ref{lemma-ring-simplicial} and from the relevant definitions.  
\end{proof} 

The following lemma (Lemma~\ref{lemma-RRk-description}) will not be used in the remainder of the paper, but the reader may find it informative. We remark that, for each integer $k \geq 4$, Lemmas~\ref{lemma-detect-ring},~\ref{lemma-simplicial-list-alg}, and~\ref{lemma-RRk-description} readily yield $O(n^3)$ time recognition algorithms for the classes $\mathcal{R}_k$ and $\mathcal{R}_{\geq k}$. However, we will not need these algorithms in the remainder of the paper, and so we leave the details to the reader. 

\begin{lemma} \label{lemma-RRk-description} Let $k \geq 4$ be an integer, and let $G$ be a graph. Then the following are equivalent: 
\begin{itemize} 
\item[(a)] $G \in \mathcal{R}_k$; 
\item[(b)] either $G$ is chordal, or $G$ is a $k$-ring, or $G$ can be obtained from a $k$-ring by repeatedly adding simplicial vertices. 
\end{itemize} 
\end{lemma} 
\begin{proof}
Suppose first that (a) holds, i.e.\ that $G \in \mathcal{R}_k$. Let $v_1,\dots,v_t$ ($t \geq 0$) be a maximal sequence of pairwise distinct vertices of $G$ such that for all $i \in \{1,\dots,t\}$, $v_i$ is simplicial in $G \setminus \{v_1,\dots,v_{i-1}\}$. If $V(G) = \{v_1,\dots,v_t\}$, then $v_1,\dots,v_t$ is a simplicial elimination ordering of $G$, and so (by~\cite{FulkersonGross}) $G$ is chordal. Suppose now that $\{v_1,\dots,v_t\} \subsetneqq V(G)$. Set $R = G \setminus \{v_1,\dots,v_t\}$. Since $G \in \mathcal{R}_k$, and since $\mathcal{R}_k$ is hereditary, we see that $R \in \mathcal{R}_k$. On the other hand, by the maximality of $v_1,\dots,v_t$, we know that $R$ has no simplicial vertices. So, by the definition of $\mathcal{R}_k$, $R$ is a $k$-ring. If $t = 0$, then $G = R$, and we have that $G$ is a $k$-ring. On the other hand, if $t \geq 1$, then $G$ can be obtained from the $k$-ring $R$ by adding simplicial vertices $v_t,\dots,v_1$ (in that order). So, (b) holds. 

Suppose now that (b) holds. Clearly, every induced subgraph of a chordal graph is chordal. Furthermore, by~\cite{FulkersonGross}, every chordal graph has a simplicial vertex. So, if $G$ is chordal, then all its induced subgraphs contain a simplicial vertex, and it follows that $G \in \mathcal{R}_k$. Suppose now that $G$ can be obtained from a $k$-ring by (possibly) repeatedly adding simplicial vertices. But then Lemma~\ref{lemma-ring-simplicial} implies that every induced subgraph of $G$ either is a $k$-ring or has a simplicial vertex, and so $G \in \mathcal{R}_k$. Thus, (a) holds. 
\end{proof} 

\begin{lemma} \label{lemma-simplicial-chi} Let $G$ be a graph on at least two vertices, and let $v$ be a simplicial vertex of $G$. Then $\omega(G) = \max\{|N_G[v]|,\omega(G \setminus v)\}$ and $\chi(G) = \max\{\omega(G),\chi(G \setminus v)\}$. 
\end{lemma} 
\begin{proof} 
We first show that $\omega(G) = \max\{|N_G[v]|,\omega(G \setminus v)\}$. Since $v$ is simplicial, $N_G[v]$ is a clique, and we deduce that $\max\{|N_G[v]|,\omega(G \setminus v)\} \leq \omega(G)$. To prove the reverse inequality, let $K$ be a clique of size $\omega(G)$ in $G$. If $v \notin K$, then $K$ is a clique of $G \setminus v$, and so $\omega(G) = |K| \leq \omega(G \setminus v) \leq \max\{|N_G[v]|,\omega(G \setminus v)\}$. So suppose that $v \in K$. Since $K$ is a clique, it follows that $K \subseteq N_G[v]$, and so $\omega(G) = |K| \leq |N_G[v]| \leq \max\{|N_G[v]|,\omega(G \setminus v)\}$. This proves that $\omega(G) = \max\{|N_G[v]|,\omega(G \setminus v)\}$. 

It remains to show that $\chi(G) = \max\{\omega(G),\chi(G \setminus v)\}$. It is clear that $\max\{\omega(G),\chi(G \setminus v)\} \leq \chi(G)$. For the reverse inequality, we set $\ell = \max\{\omega(G),\chi(G \setminus v)\}$, and we construct a proper coloring of $G$ that uses at most $\ell$ colors. First, we properly color $G \setminus v$ with colors $1,\dots,\ell$. Next, since $N_G[v]$ is a clique, we see that $|N_G(v)| = |N_G[v]|-1 \leq \omega(G)-1 \leq \ell-1$; thus, at least one of our $\ell$ colors was not used on $N_G(v)$, and we can assign this ``unused'' color to $v$. This produces a proper coloring of $G$ that uses at most $\ell$ colors, and we are done. 
\end{proof} 

We complete this section by stating the decomposition theorem for the class $\mathcal{G}_{\text{T}}$ proven in~\cite{VIK} (this is Theorem~1.8 from~\cite{VIK}). 

\begin{theorem} \label{thm-GT-decomp} \cite{VIK} Let $G \in \mathcal{G}_{\text{T}}$. Then one of the following holds: 
\begin{itemize} 
\item $G$ is a complete graph, a ring, or a 7-hyperantihole; 
\item $G$ admits a clique-cutset. 
\end{itemize} 
\end{theorem} 

Finally, we remark that graphs in $\mathcal{G}_{\text{T}}$ can be recognized in $O(n^3)$ time (see Theorem~8.23 from~\cite{VIK}), but we do not need this result in the remainder of the paper.

\section{Proof of Theorem~\ref{thm-ring-hyperhole}} \label{sec:main}

In this section, we prove Theorem~\ref{thm-ring-hyperhole}. We begin with an easy lemma. 

\begin{lemma} \label{lemma-ring-large-omega} Let $R$ be a $k$-ring (with $k \geq 4$) such that $\chi(R) = \omega(R)$. Then $R$ contains a $k$-hyperhole $H$ such that $\chi(H) = \chi(R)$. 
\end{lemma} 
\begin{proof} 
Let $(X_1,\dots,X_k)$ be a ring partition of $R$, and for all $i \in \{1,\dots,k\}$, let $X_i = \{u_i^1,\dots,u_i^{|X_i|}\}$ be an ordering of $X_i$ such that $X_i \subseteq N_R[u_i^{|X_i|}] \subseteq \dots \subseteq N_R[u_i^1] = X_{i-1} \cup X_i \cup X_{i+1}$, as in the definition of a ring. Let $Q$ be a clique of size $\omega(R)$ in $R$. By the definition of a ring, and by symmetry, we may assume that $Q \subseteq X_1 \cup X_2$. Since $u_1^1$ is complete to $X_2$, and since $u_2^1$ is complete to $X_1$, the maximality of $Q$ guarantees that $u_1^1,u_2^1 \in Q$, and in particular, $Q$ intersects both $X_1$ and $X_2$. Set $H = R[Q \cup \{u_3^1,u_4^1,\dots,u_k^1\}]$. Clearly, $H$ is a $k$-hyperhole. Furthermore, we have that $\omega(R) = |Q| \leq \omega(H) \leq \chi(H) \leq \chi(R)$; since $\chi(R) = \omega(R)$, it follows that $\chi(H) = \chi(R)$. 
\end{proof} 

In view of Lemma~\ref{lemma-ring-large-omega}, our next lemma (Lemma~\ref{lemma-even-ring-col-alg}) shows that Theorem~\ref{thm-ring-hyperhole} holds for even rings. We will also rely on Lemma~\ref{lemma-even-ring-col-alg} in our coloring algorithm for rings in section~\ref{sec:col}. 

\begin{lemma} \label{lemma-even-ring-col-alg} Even rings are perfect.\footnote{We remind the reader that a graph is {\em perfect} if all its induced subgraphs $H$ satisfy $\chi(H) = \omega(H)$. In particular, every perfect graph $G$ satisfies $\chi(G) = \omega(G)$. The fact that even rings are perfect easily follows from the Strong Perfect Graph Theorem~\cite{SPGT}. However, here we give an elementary proof of this fact.} Furthermore, there exists an algorithm with the following specifications: 
\begin{itemize} 
\item Input: A graph $G$; 
\item Output: Either an optimal coloring of $G$, or the true statement that $G$ is not an even ring; 
\item Running time: $O(n^3)$. 
\end{itemize} 
\end{lemma} 
\begin{proof} 
We begin by constructing the algorithm. We first call the algorithm from Lemma~\ref{lemma-detect-ring-ord} with input $G$; this takes $O(n^2)$ time. If the algorithm returns the answer that $G$ is not a ring, then we return the answer that $G$ is not an even ring, and we stop. So from now on, we assume that the algorithm returned all the following: 
\begin{itemize} 
\item the true statement that $G$ is a ring; 
\item the length $k$ and a ring partition $(X_1,\dots,X_k)$ of the ring $G$; 
\item for each $i \in \{1,\dots,k\}$, an ordering $X_i = \{u_i^1,\dots,u_i^{|X_i|}\}$ of $X_i$ such that $X_i \subseteq N_G[u_i^{|X_i|}] \subseteq \dots \subseteq N_G[u_i^1] = X_{i-1} \cup X_i \cup X_{i+1}$. 
\end{itemize} 
If $k$ is odd, then we return the answer that $G$ is not an even ring, and we stop. So assume that $k$ is even. Since $G$ is a ring, Lemma~\ref{lemma-ring-chordal}(d) guarantees that $G \in \mathcal{G}_{\text{T}}$, and so we can compute $\omega(G)$ by running the algorithm from Lemma~\ref{lemma-omega-ring} with input $G$; this takes $O(n^3)$ time. We now color $G$ as follows. For all odd $i \in \{1,\dots,k\}$ and all $j \in \{1,\dots,|X_i|\}$, we assign color $j$ to the vertex $u_i^j$; and for all even $i \in \{1,\dots,k\}$ and all $j \in \{1,\dots,|X_i|\}$, we assign color $\omega(G)-j+1$ to the vertex $u_i^j$. Since $|X_i| \leq \omega(G)$ for all $i \in \{1,\dots,k\}$, we see that our coloring uses only colors $1,\dots,\omega(G)$. Let us show that the coloring is proper. Suppose otherwise. By Lemma~\ref{lemma-ring-char}(b), there exist some $i \in \{1,\dots,k\}$, $j \in \{1,\dots,|X_i|\}$, and $\ell \in \{1,\dots,|X_{i+1}|\}$ such that $u_i^j$ and $u_{i+1}^{\ell}$ are adjacent in $G$ and were assigned the same color. Since $u_i^j$ and $u_{i+1}^{\ell}$ are adjacent, we see that $\{u_i^1,\dots,u_i^j\}$ and $\{u_{i+1}^1,\dots,u_{i+1}^{\ell}\}$ are cliques, complete to each other;\footnote{This follows from the properties of our orderings of $X_i$ and $X_{i+1}$.} thus, $\{u_i^1,\dots,u_i^j\} \cup \{u_{i+1}^1,\dots,u_{i+1}^{\ell}\}$ is a clique, and consequently, $j+\ell \leq \omega(G)$. On the other hand, by construction, we have that: 
\begin{itemize} 
\item if $i$ is odd, then $u_i^j$ received color $j$, and $u_{i+1}^{\ell}$ received color $\omega(G)-\ell+1$; 
\item if $i$ is even, then $u_i^j$ received color $\omega(G)-j+1$, and $u_{i+1}^{\ell}$ received color $\ell$. 
\end{itemize} 
Since vertices $u_i^j$ and $u_{i+1}^{\ell}$ received the same color, it follows that either $j = \omega(G)-\ell+1$ or $\omega(G)-j+1 = \ell$; in either case, we get that $j+\ell = \omega(G)+1$, contrary to the fact that $j+\ell \leq \omega(G)$. This proves that our coloring of $G$ is indeed proper. Furthermore, as pointed out above, this coloring uses at most $\omega(G)$ colors. Since $\omega(G) \leq \chi(G)$, we deduce that our coloring is optimal, and that $\chi(G) = \omega(G)$. We now return this coloring of $G$, and we stop. 

Clearly, the algorithm is correct, and its running time is $O(n^3)$. Note, furthermore, that we have established that all even rings $R$ satisfy $\chi(R) = \omega(R)$. The fact that even rings are perfect now follows from Lemmas~\ref{lemma-ring-simplicial} and~\ref{lemma-simplicial-chi} by an easy induction. 
\end{proof} 

As we pointed out above, Lemmas~\ref{lemma-ring-large-omega} and~\ref{lemma-even-ring-col-alg} together imply that even rings satisfy Theorem~\ref{thm-ring-hyperhole}. We devote the remainder of the section to proving Theorem~\ref{thm-ring-hyperhole} for odd rings. 

Given a graph $G$, a coloring $c$ of $G$, and distinct colors $a,b$ used by $c$, we set $T^{a,b}_{G,c} = G[\{x \in V(G) \mid \text{$c(x) = a$ or $c(x) = b$\}}]$;\footnote{Thus, $T^{a,b}_{G,c}$ is the subgraph of $G$ induced by the vertices colored $a$ or $b$.} note that if $c$ is a proper coloring of $G$, then $T^{a,b}_{G,c}$ is a bipartite graph, and if, in addition, $G$ contains no even holes, then $T^{a,b}_{G,c}$ is a forest. After introducing a few more definitions, we describe the structure of the components $Q$ of $T^{a,b}_{G,c}$ when $G$ is an induced subgraph of an odd ring (see Lemma~\ref{lemma-Rab}). 

We now need a few more definitions. Let $k \geq 5$ be an odd integer, let $R$ be a $k$-ring with ring partition $(X_1,\dots,X_k)$, and for each $i \in \{1,\dots,k\}$, let $X_i = \{u_i^1,\dots,u_i^{|X_i|}\}$ be an ordering of $X_i$ such that $X_i \subseteq N_R[u_i^{|X_i|}] \subseteq \dots \subseteq N_R[u_i^1] = X_{i-1} \cup X_i \cup X_{i+1}$, as in the definition of a ring. For all $i \in \{1,\dots,k\}$ and $j,\ell \in \{1,\dots,|X_i|\}$ such that $j \leq \ell$ (resp. $j < \ell$), we say that $u_i^j$ is {\em lower} (resp.\ {\em strictly lower}) than $u_i^{\ell}$, and that $u_i^{\ell}$ is {\em higher} (resp.\ {\em strictly higher}) than $u_i^j$; under these circumstances, we also write $u_i^j \leq u_i^{\ell}$ (resp.\ $u_i^j < u_i^{\ell}$) and $u_i^{\ell} \geq u_i^j$ (resp.\ $u_i^{\ell} > u_i^j$). For each $i \in \{1,\dots,k\}$ let $s_i = u_i^1$ and $t_i = u_i^{|X_i|}$.\footnote{Thus, $s_i$ is the lowest and $t_i$ the highest vertex in $X_i$. Note that this means that $s_i$ is a highest-degree and $t_i$ a lowest-degree vertex in $X_i$.} Further, suppose that $c$ is a proper coloring of $R \setminus t_2$. For all $X \subseteq V(R) \setminus \{t_2\}$, set $c(X) = \{c(x) \mid x \in X\}$. Given distinct colors $a,b \in c(V(R) \setminus \{t_2\})$ and an index $i \in \{1,\dots,k\}$, we say that $a$ is {\em lower} than $b$ in $X_i$ with respect to $c$, and that $b$ is {\em higher} than $a$ in $X_i$ with respect to $c$, provided that either 
\begin{itemize} 
\item $a \in c(X_i \setminus \{t_2\})$ and $b \notin c(X_i \setminus \{t_2\})$,\footnote{Obviously, if $i \neq 2$, then $X_i \setminus \{t_2\} = X_i$.} or 
\item there exist indices $j,\ell \in \{1,\dots,|X_i|\}$ such that $j < \ell$, $c(u_i^j) = a$, and $c(u_i^{\ell}) = b$. 
\end{itemize} 
Let $c_1 = c(s_1)$.\footnote{Note that this means that $c_1 \notin c(X_2 \setminus \{t_2\})$. This is because $c(s_1) = c_1$, $s_1$ is complete to $X_2$ in $R$, and $c$ is a proper coloring of $R \setminus t_2$.} We say that $c$ is {\em unimprovable} if for all colors $a \in c(V(R) \setminus \{t_2\})$ such that $a \neq c_1$, and all components $Q$ of $T^{c_1,a}_{R \setminus t_2,c}$ that do not contain $s_1$, both the following are satisfied: 
\begin{itemize} 
\item for all odd $i \in \{3,\dots,k\}$ such that $Q$ intersects $X_i$, $c_1$ is lower than $a$ in $X_i$ with respect to $c$; 
\item for all even $i \in \{3,\dots,k\}$ such that $Q$ intersects $X_i$, $c_1$ is higher than $a$ in $X_i$ with respect to $c$. 
\end{itemize} 
We remark that if $c$ is an unimprovable coloring of $R \setminus t_2$, then by definition, $c$ is a proper coloring of $R \setminus t_2$, but it need not be an optimal coloring of $R \setminus t_2$, i.e.\ it may possibly use more than $\chi(R \setminus t_2)$ colors. 

\begin{lemma} \label{lemma-Rab} Let $k \geq 5$ be an odd integer, let $R$ be a $k$-ring with ring partition $(X_1,\dots,X_k)$, and for each $i \in \{1,\dots,k\}$, let $X_i = \{u_i^1,\dots,u_i^{|X_i|}\}$ be an ordering of $X_i$ such that $X_i \subseteq N_R[u_i^{|X_i|}] \subseteq \dots \subseteq N_R[u_i^1] = X_{i-1} \cup X_i \cup X_{i+1}$. Let $G$ be an induced subgraph of $R$, let $c$ be a proper coloring of $G$, let $a,b$ be distinct colors used by $c$, and let $Q$ be any component of $T^{a,b}_{G,c}$. Then there are integers $i,j \in \{1,\dots,k\}$ such that $V(Q)\subseteq X_i\cup X_{i+1}\cup\cdots\cup X_{j-1}\cup X_j$,\footnote{As usual, indices are understood to be modulo $k$.} and such that $Q$ consists of an induced path $p_i,\dots,p_j$, where $p_{\ell}\in X_{\ell}$ for all $\ell \in \{i, \dots, j\}$, plus, optionally for each $\ell \in \{i,\dots,j\}$, a vertex $p'_{\ell}\in X_{\ell}$, strictly higher than $p_{\ell}$ in $X_{\ell}$,\footnote{So, if $p_{\ell}'$ exists, then it is dominated by $p_{\ell}$ in $R$.} with $N_Q(p'_{\ell})=\{p_{\ell}\}$. 
\end{lemma} 
\begin{proof} 
By Lemma~\ref{lemma-ring-chordal}, all holes in $R$ are of length $k$, and in particular, $R$ contains no even holes. The result now readily follows from the relevant definitions. 
\end{proof} 

Our next lemma shows that any proper coloring of $R \setminus t_2$ (where $R$ and $t_2$ are as above) can be turned into an unimprovable coloring that uses no more colors than the original coloring of $R \setminus t_2$.\footnote{In particular, this implies that if $R \setminus t_2$ is $r$-colorable, then there exists an unimprovable coloring of $R \setminus t_2$ that uses at most $r$ colors.}

\begin{lemma} \label{lemma-unimprov-alg} There exists an algorithm with the following specifications: 
\begin{itemize} 
\item Input: An odd ring $R$ with ring partition $(X_1,\dots,X_k)$, for each $i \in \{1,\dots,k\}$, an ordering $X_i = \{u_i^1,\dots,u_i^{|X_i|}\}$ of $X_i$ such that $X_i \subseteq N_R[u_i^{|X_i|}] \subseteq \dots \subseteq N_R[u_i^1] = X_{i-1} \cup X_i \cup X_{i+1}$, and a proper coloring $c$ of $R \setminus u_2^{|X_2|}$; 
\item Output: An unimprovable coloring of $R \setminus u_2^{|X_2|}$ that uses no more colors than $c$ does; 
\item Running time: $O(n^4)$. 
\end{itemize} 
\end{lemma} 
\begin{proof} 
To simplify notation, for all $i \in \{1,\dots,k\}$, we set $s_i = u_i^1$ and $t_i = u_i^{|X_i|}$. (Thus, $c$ is a proper coloring of $R \setminus t_2$.) Let $r$ be the number of colors used by $c$; by symmetry, we may assume that $c:V(R) \setminus \{t_2\} \rightarrow \{1,\dots,r\}$. Set $c_1 = c(s_1)$. 

Now, for every proper coloring $\widetilde{c}:V(R) \setminus \{t_2\} \rightarrow \{1,\dots,r\}$ of $R \setminus t_2$ such that $\widetilde{c}(s_1) = c_1$,\footnote{Note that this implies that $c_1 \notin \widetilde{c}(X_2 \setminus \{t_2\})$. This is because $\widetilde{c}(s_1) = c_1$, $s_1$ is complete to $X_2$ in $R$, and $\widetilde{c}$ is a proper coloring of $R \setminus t_2$.} we define the {\em rank} of $\widetilde{c}$, denoted by $\text{rank}(\widetilde{c})$, as follows. 
\begin{itemize} 
\item For all odd $i \in \{3,\dots,k\}$, if there exists an index $j \in \{1,\dots,|X_i|\}$ such that $\widetilde{c}(u_i^j) = c_1$,\footnote{Note that if $j$ exists, then it is unique. This is because $X_i$ is a clique of $R \setminus t_2$, and $\widetilde{c}$ is a proper coloring of $R \setminus t_2$.} then we set $r_i(\widetilde{c}) = j$, and otherwise, we set $r_i(\widetilde{c}) = |X_i|+1$. 
\item For all even $i \in \{3,\dots,k\}$, if there exists an index $j \in \{1,\dots,|X_i|\}$ such that $\widetilde{c}(u_i^j) = c_1$,\footnote{As before, if $j$ exists, then it is unique.} then we set $r_i(\widetilde{c}) = |X_i|-j+2$, and otherwise, we set $r_i(\widetilde{c}) = 1$. 
\item We set $\text{rank}(\widetilde{c}) = \sum\limits_{i=3}^k r_i(\widetilde{c})$.\footnote{Note that $k-2 \leq \text{rank}(\widetilde{c}) \leq k-2+\sum\limits_{i=3}^k |X_i|$. So, rank can take at most $1+\sum\limits_{i=3}^k |X_i| < n$ different values.} 
\end{itemize} 

The algorithm proceeds as follows. We check whether the input coloring $c$ is unimprovable by examining all colors $a \in \{1,\dots,r\} \setminus \{c_1\}$, and all components $Q$ of $T^{c_1,a}_{R \setminus t_2,c}$ that do not contain $s_1$; this can be done in $O(n^3)$ time. If $c$ is unimprovable, then we return $c$, and we stop. Otherwise, the algorithm found some color $a \in \{1,\dots,r\} \setminus \{c_1\}$, some component $Q$ of $T^{c_1,a}_{R \setminus t_2,c}$ that does not contain $s_1$, and some index $i^* \in \{3,\dots,k\}$ such that $Q$ intersects $X_{i^*}$ and either 
\begin{itemize} 
\item $i^*$ is odd, and $a$ is lower than $c_1$ in $X_{i^*}$ with respect to $c$; or 
\item $i^*$ is even, and $a$ is higher than $c_1$ in $X_{i^*}$ with respect to $c$. 
\end{itemize} 
Lemma~\ref{lemma-Rab} then implies that both the following hold: 
\begin{itemize} 
\item for all odd $i \in \{3,\dots,k\}$ such that $Q$ intersects $X_i$, $a$ is lower than $c_1$ in $X_i$ with respect to $c$; 
\item for all even $i \in \{3,\dots,k\}$ such that $Q$ intersects $X_i$, $a$ is higher than $c_1$ in $X_i$ with respect to $c$. 
\end{itemize} 
Let $c'$ be the coloring of $R \setminus t_2$ obtained from $c$ by swapping colors $c_1$ and $a$ on $Q$.\footnote{Since $Q$ does not contain $s_1$, we have that $c'(s_1) = c(s_1) = c_1$.} Note that $\text{rank}(c') < \text{rank}(c)$. We now update the coloring $c$ by setting $c := c'$, and we obtain an unimprovable coloring of $R \setminus t_2$ by making a recursive call to the algorithm. 

The algorithm terminates because the rank of the coloring $c$ decreases before each recursive call. We make $O(n)$ recursive calls,\footnote{This is because rank can take at most $n$ different values, and the rank of our coloring decreases before each recursive call.} and otherwise, the slowest step of the algorithm takes $O(n^3)$ time. So, the total running time of the algorithm is $O(n^4)$. 
\end{proof} 

We now prove a technical lemma (Lemma~\ref{lemma-main-technical}) that is at the heart of our proof of Theorem~\ref{thm-ring-hyperhole} for odd rings. We also rely on Lemma~\ref{lemma-main-technical} in our coloring algorithm for rings.\footnote{More precisely, our coloring algorithm for rings relies on Lemma~\ref{lemma-reduce-chi}, which is an easy corollary of Lemma~\ref{lemma-main-technical} and Theorem~\ref{thm-ring-hyperhole}.} We remark that in our proof of Lemma~\ref{lemma-main-technical}, we repeatedly rely on Lemma~\ref{lemma-Rab} without explicitly stating this.\footnote{Essentially, every time we consider a component $Q$ as in Lemma~\ref{lemma-Rab}, we keep in mind the structure of $Q$, as described in Lemma~\ref{lemma-Rab}.} 

\begin{lemma} \label{lemma-main-technical} Let $k \geq 5$ be an odd integer, let $R$ be a $k$-ring with ring partition $(X_1,\dots,X_k)$, and for each $i \in \{1,\dots,k\}$, let $X_i = \{u_i^1,\dots,u_i^{|X_i|}\}$ be an ordering of $X_i$ such that $X_i \subseteq N_R[u_i^{|X_i|}] \subseteq \dots \subseteq N_R[u_i^1] = X_{i-1} \cup X_i \cup X_{i+1}$. For all $i \in \{1,\dots,k\}$, set $s_i = u_i^1$ and $t_i = u_i^{|X_i|}$. Let $c$ be an unimprovable coloring of $R \setminus t_2$, and let $r$ be the number of colors used by $c$.\footnote{In particular, $c$ is a proper coloring of $R \setminus t_2$. Furthermore, we have that $\chi(R \setminus t_2) \leq r$, and this inequality may possibly be strict.} Let $c_1 = c(s_1)$, and let $S = \{x \in V(R) \mid x \neq t_2, c(x) = c_1\}$.\footnote{Note that $S$ is a stable set in $R \setminus t_2$. Furthermore, $s_1 \in S$, and in particular, $S \neq \emptyset$.} Then both the following hold: 
\begin{itemize} 
\item[(a)] either $\omega(R \setminus S) \leq r-1$, or $R$ contains a $k$-hyperhole of chromatic number $r+1$; 
\item[(b)] if every $k$-ring $R'$ such that $|V(R')| < |V(R)|$ contains a $k$-hyperhole of chromatic number $\chi(R')$, then either $\chi(R \setminus S) \leq r-1$, or $R$ contains a $k$-hyperhole of chromatic number $r+1$. 
\end{itemize} 
\end{lemma} 
\begin{proof} 
By hypotheses, we have that $\chi(R \setminus t_2) \leq r$; it follows that $\omega(R) \leq \chi(R) \leq r+1$. If $\omega(R) = r+1$, then both (a) and (b) follow from Lemma~\ref{lemma-ring-large-omega}; thus, we may assume that $\omega(R) \leq r$. 

Set $Y_1 = N_R(t_2) \cap X_1$, $X_2' = X_2 \setminus \{t_2\}$, and $Y_3 = N_R(t_2) \cap X_3$. Note that $N_R(t_2) = Y_1 \cup X_2' \cup Y_3$, with $Y_1,X_2',Y_3$ pairwise disjoint. Furthermore, we have that $s_1 \in Y_1$ and $s_3 \in Y_3$, and in particular, $Y_1$ and $Y_3$ are nonempty (the set $X_2'$ may possibly be empty). Finally, we remark that $Y_1 \cup X_2$ and $X_2 \cup Y_3$ are maximal cliques of $R$. 

Let $C$ be the set of colors used by $c$; then $|C| = r$. To simplify notation, for all distinct colors $a,b \in C$, we write $T^{a,b}$ instead of $T^{a,b}_{R \setminus t_2,c}$. Further, for all $i \in \{1,\dots,k\} \setminus \{2\}$ and $a \in c(X_i)$, we denote by $x_i^a$ the (unique) vertex of $X_i$ to which $c$ assigned color $a$; similarly, for all $a \in c(X_2')$, we denote by $x_2^a$ the (unique) vertex of $X_2'$ to which $c$ assigned color $a$. Finally, when we say that some color is higher or lower than some other color in some $X_i$, we always mean this with respect to our coloring $c$.

\begin{quote} 
\begin{claim} \label{claim-a} Either $\omega(R \setminus S) \leq r-1$, or $R$ contains a $k$-hyperhole of chromatic number $r+1$. In other words, (a) holds. 
\end{claim} 
\end{quote} 
{\em Proof of Claim~\ref{claim-a}.} 
Since $\omega(R) \leq r$, we have that $\omega(R \setminus S) \leq r$. Thus, we may assume that $\omega(R \setminus S) = r$, for otherwise we are done; since $\omega(R) \leq r$, this implies that $\omega(R) = r$. 

Since $c$ is a proper coloring of $R \setminus t_2$ that uses only $r$ colors, and since $S$ is a color class of the coloring $c$, we see that $S$ intersects all cliques of size $r$ in $R$ that do not contain $t_2$. Furthermore, there are exactly two maximal cliques in $R$ that contain $t_2$, namely $Y_1 \cup X_2$ and $X_2 \cup Y_3$. Since $S$ intersects $Y_1 \cup X_2$ (because $s_1 \in Y_1 \cap S$), we deduce that $X_2 \cup Y_3$ is the unique clique of $R \setminus S$ of size $r$. (Note that this implies that $X_2' \cup Y_3$ is a clique of size $r-1$.) In particular, $c_1 \notin c(X_2' \cup Y_3)$. 

Consider any color $a \in c(Y_3)$, and let $Q$ be the component of $T^{c_1,a}$ that contains the vertex of $Y_3$ colored $a$. Since $c_1 \notin c(Y_3)$, we see that $a \neq c_1$, and furthermore, $a$ is lower than $c_1$ in $X_3$. So, since $c$ is unimprovable, we have that $s_1 \in V(Q)$. Further, since $c_1 \notin c(X_2')$, we see that $V(Q) \cap X_2' = \emptyset$. We now deduce that the following hold: 
\begin{itemize} 
\item for every odd $i \neq 1$, we have that $c(Y_3) \subseteq c(X_i)$; 
\item for every even $i \neq 2$, some vertex of $X_i$ is colored $c_1$,\footnote{Recall that this vertex is called $x_i^{c_1}$.} and furthermore, this vertex is adjacent to all vertices of $X_{i-1} \cup X_{i+1}$ that received a color used on $Y_3$. 
\end{itemize} 

For odd $i \geq 5$, let $h_i$ be the highest vertex of $X_i$ that is adjacent both to $x_{i-1}^{c_1}$ and to $x_{i+1}^{c_1}$.\footnote{Let us check that such an $h_i$ exists. Since $i \geq 5$ is odd, we see that either $5 \leq i \leq k-2$ or $i = k$. If $5 \leq i \leq k-2$, then $i-1,i+1 \geq 4$ are both even, and so by what we just showed, $x_{i-1}^{c_1}$ and $x_{i+1}^{c_1}$ are both defined. If $i = k$, then once again, $i-1 \geq 4$ is even, and so $x_{i-1}^{c_1}$ is defined, and furthermore, since our subscripts are understood to be modulo $k$, we have that $x_{i+1}^{c_1} = x_1^{c_1} = s_1$. So, in either case, $x_{i-1}^{c_1}$ and $x_{i+1}^{c_1}$ are both defined. Moreover, at least one vertex of $X_i$ (namely, the vertex $s_i$) is adjacent both to $x_{i-1}^{c_1}$ and to $x_{i-1}^{c_1}$. So, $h_i$ exists.} We now define sets $Z_1,\dots,Z_k$ as follows: 
\begin{itemize} 
\item let $Z_1 = \{s_1\}$, $Z_2 = X_2$, and $Z_3 = Y_3$; 
\item for all even $i \geq 4$, let $Z_i = \{x \in X_i \mid x \leq x_i^{c_1}\}$; 
\item for all odd $i \geq 5$, let $Z_i = \{x \in X_i \mid x \leq h_i\}$. 
\end{itemize} 
Finally, let $H = R[Z_1 \cup Z_2 \cup \dots \cup Z_k]$. 

By construction, $H$ is a $k$-hyperhole of $R$; thus, $\chi(H) \leq \chi(R) \leq r+1$. If $\chi(H) = r+1$, then we are done. So assume that $\chi(H) \leq r$. Then $\Big\lceil \frac{2|V(H)|}{k-1} \Big\rceil = \Big\lceil \frac{|V(H)|}{\alpha(H)} \Big\rceil \leq \chi(H) \leq r$. It follows that $|V(H)| \leq \frac{k-1}{2}r$, and consequently, $|V(H) \setminus \{t_2\}| < \frac{k-1}{2}r$. Now, $X_2' \cup Y_3$ is a clique of size $r-1$ in $R \setminus t_2$, and so $|c(X_2' \cup Y_3)| = r-1$. Furthermore, we know that $c_1 \notin c(X_2' \cup Y_3)$, and so $|\{c_1\} \cup c(X_2' \cup Y_3)| = r$. Since $|V(H) \setminus \{t_2\}| < \frac{k-1}{2}r$, we see that some color from $\{c_1\} \cup c(X_2' \cup Y_3)$ appears on fewer than $\frac{k-1}{2}$ vertices of $H \setminus t_2$. Now, by construction, every color from $\{c_1\} \cup c(Y_3)$ appears $\frac{k-1}{2}$ times on $H \setminus t_2$. It follows that some color $d \in c(X_2')$ appears fewer than $\frac{k-1}{2}$ times on $H \setminus t_2$. Thus, there exists some even $i \geq 4$ such that $d \notin c(Z_i)$;\footnote{By the construction of $Z_i$, this implies that $d \neq c_1$, and that $d$ is higher than $c_1$ in $X_i$.} let $i$ be the smallest such index. Thus, $d$ appears on each $Z_j$, for even $j < i$, and there are $\frac{i}{2}-1$ such $j$'s. On the other hand, let $Q$ be the component of $T^{c_1,d}$ that contains $x_i^{c_1}$. Now, we have that $i \geq 4$ is even, and that $d$ is higher than $c_1$ in $X_i$; since $c$ is unimprovable, we deduce that $s_1 \in V(Q)$. It follows that each $Z_j$, for odd $j > i$, contains a vertex colored $d$; there are $\lceil \frac{k-i}{2} \rceil = \frac{k-i+1}{2}$ such $j$'s. So, in total, at least $(\frac{i}{2}-1)+\frac{k-i+1}{2} = \frac{k-1}{2}$ vertices of $H \setminus t_2$ are colored $d$, contrary to our choice of $d$.~$\blacksquare$ 

\medskip 

It remains to prove (b). For this, we assume that both the following hold: 
\begin{itemize} 
\item every $k$-ring $R'$ such that $|V(R')| < |V(R)|$ contains a $k$-hyperhole of chromatic number $\chi(R')$; 
\item $\chi(R \setminus S) \geq r$; 
\end{itemize} 
and we prove that $R$ contains a $k$-hyperhole of chromatic number $r+1$. 

Since $S$ is a color class of a proper coloring of $R \setminus t_2$ that uses at most $r$ colors, we see that $\chi\Big(R \setminus (S \cup \{t_2\})\Big) \leq r-1$; consequently, $\chi(R \setminus S) \leq r$. Since $\chi(R \setminus S) \geq r$, it follows that $\chi(R \setminus S) = r$. Further, in view of (a), we may assume that $\omega(R \setminus S) \leq r-1$. 

\begin{quote} 
\begin{claim} \label{claim-R-S-hyperhole} $R \setminus S$ contains a $k$-hyperhole $H$ such that $\chi(H) = \Big\lceil \frac{2|V(H)|}{k-1} \Big\rceil = r$. 
\end{claim} 
\end{quote} 
{\em Proof of Claim~\ref{claim-R-S-hyperhole}.} 
Let $v_1,\dots,v_t$ (with $t \geq 0$) be a maximal sequence of pairwise distinct vertices in $R \setminus S$ such that for all $i \in \{1,\dots,t\}$, $v_i$ is simplicial in $R \setminus (S \cup \{v_1,\dots,v_{i-1}\})$. Set $A = \{v_1,\dots,v_t\}$. Suppose first that $R \setminus S = A$. Then $v_1,\dots,v_t$ is a simplicial elimination ordering of $R \setminus S$, and so by coloring $R \setminus S$ greedily using the ordering $v_t,\dots,v_1$, we obtain a proper coloring of $R \setminus S$ that uses only $\omega(R \setminus S)$ colors, contrary to the fact that $\chi(R \setminus S) = r > \omega(R \setminus S)$. So, $R \setminus S \neq A$. Lemma~\ref{lemma-ring-simplicial} and the maximality of $A$ now imply that $R \setminus (S \cup A)$ is a $k$-ring. Since $S \neq \emptyset$, the $k$-ring $R \setminus (S \cup A)$ has fewer vertices than $R$, and so $R \setminus (S \cup A)$ contains a $k$-hyperhole $H$ such that $\chi(H) = \chi\Big(R \setminus (S \cup A)\Big)$. 

Now, Lemma~\ref{lemma-simplicial-chi} and an easy induction guarantee that 
\begin{displaymath} 
\chi(R \setminus S) = \max\Big\{\omega(R \setminus S),\chi\Big(R \setminus (S \cup A)\Big)\Big\}. 
\end{displaymath} 
Since $\chi(R \setminus S) = r$, $\omega(R \setminus S) \leq r-1$, and $\chi(H) = \chi\Big(R \setminus (S \cup A)\Big)$, we deduce that $\chi(H) = r$. Since $\omega(H) \leq \omega(R \setminus S) \leq r-1$, we see that $\omega(H) < \chi(H)$, and so Lemma~\ref{lemma-hyperhole-chi-formula} implies that $\chi(H) = \Big\lceil \frac{|V(H)|}{\alpha(H)} \Big\rceil = \Big\lceil \frac{2|V(H)|}{k-1} \Big\rceil$. Thus, $\chi(H) = \Big\lceil \frac{2|V(H)|}{k-1} \Big\rceil = r$.~$\blacksquare$ 

\medskip 

From now on, let $H$ be as in Claim~\ref{claim-R-S-hyperhole}. Our goal is to find a hyperhole in $R$ of size at least $|V(H)|+\frac{k+1}{2}$; this will imply\footnote{The details are given at the end of the proof of the lemma.} that the chromatic number of that hyperhole is at least $r+1$,\footnote{Since $\chi(R) \leq r+1$, we see that any hyperhole in $R$ of chromatic number at least $r+1$ in fact has chromatic number exactly $r+1$.} which is what we need. 

For each $i \in \{1,\dots,k\}$, let $h_i$ be the highest vertex of $X_i \cap V(H)$. Then $h_1,\dots,h_k,h_1$ is a $k$-hole in $R$, and we may assume that for all $i \in \{1,\dots,k\}$, we have that $V(H) \cap X_i = \{x \in X_i \mid x \leq h_i\} \setminus S$.\footnote{Indeed, set $H' = R[\bigcup\limits_{i=1}^k (\{x \in X_i \mid x \leq h_i\} \setminus S)]$. It is clear that $H'$ is a $k$-hyperhole in $R \setminus S$, and that $H'$ contains $H$ as an induced subgraph. So, $r = \chi(H) \leq \chi(H') \leq \chi(R \setminus S) = r$, and it follows that $\chi(H') = r$. On the other hand, $\omega(H') \leq \omega(R \setminus S) \leq r-1 < \chi(H')$, and so Lemma~\ref{lemma-hyperhole-chi-formula} implies that $\chi(H') = \Big\lceil \frac{|V(H')|}{\alpha(H')} \Big\rceil = \Big\lceil \frac{2|V(H')|}{k-1} \Big\rceil$. Thus, $\chi(H') = \Big\lceil \frac{2|V(H')|}{k-1} \Big\rceil = r$. So, if $H' \neq H$, then from now on, instead of $H$, we simply consider $H'$.} In particular, we have that $s_1,\dots,s_k \in V(H) \cup S$. 

Recall that $c_1 = c(s_1)$. Let $j$ be the largest odd index such that $c(s_i) = c_1$ for all odd $i \in \{1,\dots,j\}$. Then $j \leq k-2$.\footnote{Indeed, $s_k$ and $s_1$ are adjacent, and $c(s_1) = c_1$; so, $c(s_k) \neq c_1$, and it follows that $j \neq k$. Since $j$ and $k$ are both odd, we deduce that $j \leq k-2$.} Furthermore, since $s_j$ is complete to $X_{j+1}$, we have that $c_1 \notin c(X_{j+1})$. 

\begin{quote} 
\begin{claim} \label{claim-c1-in-even} $c_1 \in c(X_i)$ for every even index $i \geq j+3$.\footnote{So, $x_i^{c_1}$ is defined for every even index $i \geq j+3$.} 
\end{claim} 
\end{quote} 
{\em Proof of Claim~\ref{claim-c1-in-even}.} 
Suppose otherwise, and fix the smallest even index $i \geq j+3$ such that $c_1 \notin c(X_i)$. If $c(s_{i-1}) = c_1$, then: 
\begin{itemize} 
\item if $i-1 = j+2$, then the choice of $j$ is contradicted; 
\item if $i-1 \geq j+4$, then the choice of $i$ is contradicted.\footnote{We are using the fact that $s_{i-1}$ is complete to $X_{i-2}$, and so $c(s_{i-1}) \notin c(X_{i-2})$.} 
\end{itemize} 
It follows that $c(s_{i-1}) \neq c_1$. Set $c_{i-1} = c(s_{i-1})$; since $s_{i-1}$ is complete to $X_i$, we have that $c_{i-1} \notin c(X_i)$. Let $Q$ be the component of $T^{c_1,c_{i-1}}$ that contains $s_{i-1}$. We know that $c_1,c_{i-1} \notin c(X_i)$, and so $V(Q) \cap X_i = \emptyset$. On the other hand, by the parity of $i$ and $j$, and by the fact that $c_1 \notin c(X_{j+1})$, we have that $V(Q) \cap X_{j+1} = \emptyset$. Thus, $V(Q) \subseteq X_{j+2} \cup \dots \cup X_{i-1}$. We now have that $s_1 \notin V(Q)$, that $i-1 \geq 3$ is odd, that $Q$ intersects $X_{i-1}$, and that $c_{i-1}$ is lower than $c_1$ in $X_{i-1}$. But this contradicts the fact that $c$ is unimprovable.~$\blacksquare$ 

\medskip 

Recall that $h_i$ be the highest vertex of $X_i \cap V(H)$. Let $\ell$ be the largest odd index such that for every odd $i \in \{1,\dots,\ell\}$, the coloring $c$ assigns color $c_1$ to some vertex of $X_i$ lower than $h_i$.\footnote{So, for all odd $i \in \{1,\dots,\ell\}$, we have that $x_i^{c_1}$ is defined and satisfies $x_i^{c_1} \leq h_i$.} Clearly, $j \leq \ell \leq k-2$.\footnote{The fact that $j \leq \ell$ is immediate from the the choice of $j$ and $\ell$. The fact that $\ell \neq k$ follows from the fact that $c(s_1) = c_1$, and that $s_1$ is complete to $X_k$, so that $c_1 \notin c(X_k)$. Since $\ell$ and $k$ are both odd, it follows that $\ell \leq k-2$.} Now, we define vertices $w_1,\dots,w_k$ as follows: 
\begin{itemize} 
\item for $i \leq \ell+2$, let $w_i = h_i$; 
\item for even $i \geq \ell+3$, let $w_i = \max\{h_i,x_i^{c_1}\}$;\footnote{Claim~\ref{claim-c1-in-even} guarantees that $c_1 \in c(X_i)$, and so $x_i^{c_1}$ is defined.} 
\item for odd $i \geq \ell+4$, let $w_i$ be the highest vertex of $X_i \cap V(H)$ that is adjacent to $x_{i-1}^{c_1}$.\footnote{Let us check that such a $w_i$ exists. First, Claim~\ref{claim-c1-in-even} guarantees that $x_{i-1}^{c_1}$ is defined. It now suffices to show that some vertex of $X_i \cap V(H)$ is adjacent to $x_{i-1}^{c_1}$. Clearly, $s_i$ is adjacent to $x_{i-1}^{c_1}$. Since $c(x_{i-1}^{c_1}) = c_1$, and since $c$ is a proper coloring of $R \setminus t_2$, we have that $c(s_i) \neq c_1$; consequently, $s_i \notin S$. Since $s_1,\dots,s_k \in V(H) \cup S$, we deduce that $s_i \in V(H)$. So, $X_i \cap V(H)$ contains a vertex (namely $s_i$) that is adjacent to $x_{i-1}^{c_1}$.} 
\end{itemize} 
For all $i \in \{1,\dots,k\}$, let $W_i = \{x \in X_i \mid x \leq w_i\}$. Further, let $W = R[W_1 \cup \dots \cup W_k]$. Finally, let $S_W = \{x \in V(W) \mid x \neq t_2, c(x) = c_1\}$. We note that, by construction, $|S_W| \geq \frac{k-1}{2}$.

\begin{quote} 
\begin{claim} \label{claim-W-hyperhole} 
$W$ is a $k$-hyperhole. 
\end{claim} 
\end{quote} 
{\em Proof of Claim~\ref{claim-W-hyperhole}.} 
Suppose otherwise. Then there exists some even $i \geq \ell+3$ such that $x_i^{c_1}$ is nonadjacent to $w_{i-1}$.\footnote{Let us justify this. By supposition, $W$ is not a $k$-hyperhole, and so there exists some $i \in \{1,\dots,k\}$ such that $w_i$ is nonadjacent to $w_{i-1}$ in $R$. By the construction of $W$, we have that $w_1,\dots,w_{\ell+2} \in V(H)$, as well as that $w_k \in V(H)$; since $H$ is a hyperhole, we deduce that $i \geq \ell+3$. Suppose that $i$ is odd (thus, $i \geq \ell+4$). By Claim~\ref{claim-c1-in-even}, $x_{i-1}^{c_1}$ is defined, and since $i \geq \ell+4$ is odd, we know that $w_i$ is adjacent to $x_{i-1}^{c_1}$. On the other hand, since $i \geq \ell+4$ is odd, we have that $w_i \in V(H)$, and so $w_i$ is adjacent to $h_{i-1}$. But since $i-1 \geq \ell+3$ is even, we have by construction that $w_{i-1} = \max\{h_{i-1},x_{i-1}^{c_1}\}$; so, $w_i$ is adjacent to $w_{i-1}$, a contradiction. This proves that $i$ is even. Since $i \geq \ell+3$ is even, we have that $w_i = \max\{h_i,x_i^{c_1}\}$. Furthermore, $i-1$ is odd, and so $w_{i-1} \in V(H)$; since $H$ is a hyperhole, it follows that $h_i$ is adjacent to $w_{i-1}$. Since $w_i$ is nonadjacent to $w_{i-1}$, we deduce that $w_i = x_i^{c_1}$, and that $x_i^{c_1}$ is nonadjacent to $w_{i-1}$.} Let $a = c(w_{i-1})$. 

Suppose that $i = \ell+3$. Then by the choice of $\ell$, no vertex in $W_{i-1} = W_{\ell+2}$ is colored $c_1$. So, $a \neq c_1$, and $c_1$ is higher than $a$ in $X_{i-1} = X_{\ell+2}$. Let $Q$ be the component of $T^{c_1,a}$ that contains $w_{i-1}$. By construction, $V(Q) \cap X_i = \emptyset$, i.e.\ $V(Q) \cap X_{\ell+3} = \emptyset$; on the other hand, by the parity of $i$ and $j$, and by the fact that $c_1 \notin c(X_{j+1})$, we see that $V(Q) \cap X_{j+1} = \emptyset$. Thus, $V(Q) \subseteq X_{j+2} \cup \dots \cup X_{\ell+2}$. We now have that $s_1 \notin V(Q)$, that $\ell+2 \geq 3$ is odd, that $Q$ intersects $X_{\ell+2}$, and that $a$ is lower than $c_1$ in $X_{\ell+2}$. But this contradicts the fact that $c$ is unimprovable. 

Thus, $i \geq \ell+5$. By construction, $x_{i-2}^{c_1}$ is adjacent to $w_{i-1}$, and so if $c_1 \in c(X_{i-1})$, then $w_{i-1} < x_{i-1}^{c_1}$. Thus, $a \neq c_1$, and $a$ is lower than $c_1$ in $X_{i-1}$. Let $Q$ be the component of $T^{c_1,a}$ that contains $w_{i-1}$. Then $V(Q) \cap X_{j+1} = V(Q) \cap X_i = \emptyset$,\footnote{As before, the fact that $V(Q) \cap X_{j+1} = \emptyset$ follows from the parity of $i$ and $j$, and from the fact that $c_1 \notin c(X_{j+1})$. The fact that $V(Q) \cap X_i = \emptyset$ follows from the fact that $w_{i-1}$ is nonadjacent to $x_i^{c_1}$.} and we deduce that $V(Q) \subseteq X_{j+2} \cup \dots \cup X_{i-1}$. But now $s_1 \notin V(Q)$, $i-1 \geq 3$ is odd, $Q$ intersects $X_{i-1}$, and $a$ is lower than $c_1$ in $X_{i-1}$; this contradicts the fact that $c$ is unimprovable.~$\blacksquare$ 

\begin{quote} 
\begin{claim} \label{claim-W-big} 
$|V(W)| \geq |V(H)|+\frac{k-1}{2}$. 
\end{claim} 
\end{quote} 
{\em Proof of Claim~\ref{claim-W-big}.} 
To simplify notation, for all $i \in \{1,\dots,k\}$, we set $H_i = V(H) \cap X_i$. Recall that $S_W = \{x \in V(W) \mid x \neq t_2, c(x) = c_1\}$. We then have that $S_W \subseteq S$, that $S_W$ is a stable set in $R \setminus t_2$, and that $V(H) \cap S_W = \emptyset$. Further, recall that $|S_W| \geq \frac{k-1}{2}$. Thus, it suffices to show that $|V(H)| \leq |V(W) \setminus S_W|$. 

By the construction of $W$, for all indices $i \in \{1,\dots,k\}$ such that either $i \leq \ell+2$ or $i$ is even, we have that $H_i \subseteq W_i \setminus S_W$. We may now assume that for some even index $i \geq \ell+3$, we have that $|W_i \setminus (H_i \cup S_W)| < |H_{i+1} \setminus W_{i+1}|$, for otherwise we are done. Since $W_i \setminus (H_i \cup S_W)$ and $H_{i+1} \setminus W_{i+1}$ are both cliques of $R \setminus t_2$, and since $c$ is a proper coloring of $R \setminus t_2$, we have that $|c(W_i \setminus (H_i \cup S_W))| < |c(H_{i+1} \setminus W_{i+1})|$; fix $a \in c(H_{i+1} \setminus W_{i+1}) \setminus c(W_i \setminus (H_i \cup S_W))$. Then $a \neq c_1$.\footnote{This is because $a \in c(H_{i+1})$, and $c$ does not assign color $c_1$ to any vertex in $V(H) \setminus \{t_2\}$.} Furthermore, we have that $a \notin c(W_i)$,\footnote{By construction, $a \notin c(W_i \setminus (H_i \cup S_W))$, and since $a \neq c_1$, we also have that $a \notin c(S_W)$. Further, $a \in c(H_{i+1})$, and so since $H_i$ is complete to $H_{i+1}$, we have that $a \notin c(H_i)$. Thus, $a \notin c(W_i)$.} whereas by the construction of $W$, and by the fact that $i \geq \ell+3$ is even, we have that $c_1 \in c(W_i)$. It then follows from the construction of $W$ that $a$ is higher than $c_1$ in $X_i$ (possibly $a \notin c(X_i)$). 

Since $a \in c(H_{i+1} \setminus W_{i+1})$, we have that $x_{i+1}^a \in H_{i+1} \setminus W_{i+1}$. Since $i+1$ is odd with $i+1 \geq \ell+4$, we see from the construction of $W$ that $x_{i+1}^a$ is nonadjacent to $x_i^{c_1}$. Let $Q$ be the component of $T^{c_1,a}$ that contains $x_i^{c_1}$. Then $V(Q) \cap X_{j+1} = V(Q) \cap X_{i+1} = \emptyset$,\footnote{The fact that $V(Q) \cap X_{i+1} = \emptyset$ follows from the fact that $a$ is higher than $c_1$ in $X_i$, and $x_i^{c_1}$ is nonadjacent to $x_{i+1}^a$. The fact that $V(Q) \cap X_{j+1} = \emptyset$ follows from the parity of $i$ and $j$, and from the fact that $c_1 \notin c(X_{j+1})$.} and it follows that $V(Q) \subseteq X_{j+2} \cup \dots \cup X_i$. We now have that $s_1 \notin V(Q)$, that $i \geq 4$ is even, that $Q$ intersects $X_i$, and that $a$ is higher than $c_1$ in $X_i$. But this contradicts the fact that $c$ is unimprovable.~$\blacksquare$ 

\medskip 

By Claim~\ref{claim-W-hyperhole}, $W$ is a $k$-hyperhole; since $k$ is odd, we see that $\alpha(W) = \frac{k-1}{2}$. Using Claims~\ref{claim-R-S-hyperhole} and~\ref{claim-W-big}, we now get that 
\begin{displaymath} 
\begin{array}{ccc ccc ccc} 
\chi(W) & \geq & \Big\lceil \frac{|V(W)|}{\alpha(W)} \Big\rceil & = & \Big\lceil \frac{2|V(W)|}{k-1} \Big\rceil & \geq & \Big\lceil \frac{2|V(H)|}{k-1} \Big\rceil+1 & = & r+1. 
\end{array} 
\end{displaymath} 
On the other hand, we have that $\chi(W) \leq \chi(R) \leq r+1$, and we deduce that $\chi(W) = r+1$. This proves (b), and we are done. 
\end{proof} 

We are now ready to prove Theorem~\ref{thm-ring-hyperhole}, restated below for the reader's convenience. 

\begin{thm-ring-hyperhole} Let $k \geq 4$ be an integer, and let $R$ be a $k$-ring. Then $\chi(R) = \max\{\chi(H) \mid \text{$H$ is a $k$-hyperhole in $R$}\}$.
\end{thm-ring-hyperhole} 
\begin{proof} 
If $k$ is even, then the result follows from Lemmas~\ref{lemma-ring-large-omega} and~\ref{lemma-even-ring-col-alg}. So from now on, we assume that $k$ is odd. Clearly, it suffices to show that $R$ contains a $k$-hyperhole of chromatic number $\chi(R)$. We assume inductively that this holds for smaller $k$-rings, i.e.\ we assume that every $k$-ring $R'$ such that $|V(R')| < |V(R)|$ contains a $k$-hyperhole of chromatic number $\chi(R')$. 

Let $(X_1,\dots,X_k)$ be a ring partition of $R$.  For each $i\in\{1, \dots , k\}$, let $X_i = \{u_i^1, \dots, u_i^{|X_i|}\}$ be an ordering of $X_i$ such that $X_i \subseteq N_R[u_i^{|X_i|}] \subseteq \dots \subseteq N_R[u_i^1] = X_{i-1} \cup X_i \cup X_{i+1}$, as in the definition of a ring. For all $i \in \{1,\dots,k\}$, set $s_i = u_i^1$ and $t_i = u_i^{|X_i|}$. Set $r = \chi(R \setminus t_2)$, and note that this implies that $r \leq \chi(R) \leq r+1$. Thus, we may assume that $R$ contains no hyperhole of chromatic number $r+1$, for otherwise we are done. 

Let $c$ be an unimprovable coloring of $R \setminus t_2$ that uses exactly $r$ colors (the existence of such a coloring follows from Lemma~\ref{lemma-unimprov-alg}). Let $C$ be the set of colors used by $c$ (thus, $|C| = r$), and set $c_1 = c(s_1)$ and $S = \{x \in V(R) \mid x \neq t_2, c(x) = c_1\}$. Lemma~\ref{lemma-main-technical} now implies that $\omega(R \setminus S) \leq r-1$ and $\chi(R \setminus S) \leq r-1$. Since $S$ is a stable set in $R$, we see that $\chi(R) \leq \chi(R \setminus S)+1 \leq r$; we already saw that $r \leq \chi(R) \leq r+1$, and so we deduce that $\chi(R) = r$. Further, since $\omega(R \setminus S) \leq r-1$, and since $S$ is a stable set, we see that $\omega(R) \leq r$. If $\omega(R) = r$, then $\chi(R) = \omega(R)$, and the result follows from Lemma~\ref{lemma-ring-large-omega}. Thus, we may assume that $\omega(R) \leq r-1$. Clearly, this implies that $\omega(R \setminus t_2) \leq r-1$. Since $\chi(R \setminus t_2) = r$, we have that $\omega(R \setminus t_2) < \chi(R \setminus t_2)$. 

Suppose that $|X_2| = 1$, i.e.\ that $X_2 = \{t_2\}$. Then by Lemma~\ref{lemma-ring-chordal}(c), $R \setminus t_2$ is chordal, and therefore (by~\cite{Berge-German, D61}) perfect. So, $\chi(R \setminus t_2) = \omega(R \setminus t_2)$, contrary to the fact that $\omega(R \setminus t_2) < \chi(R \setminus t_2)$. So, $|X_2| \geq 2$. Since every vertex in $X_2 \setminus \{t_2\}$ dominates $t_2$ in $R$, Lemma~\ref{lemma-ring-char} readily implies that $R \setminus t_2$ is a $k$-ring with ring partition $(X_1,X_2 \setminus \{t_2\},X_3,\dots,X_k)$. So, by the induction hypothesis, $R \setminus t_2$ contains a $k$-hyperhole $H$ of chromatic number $\chi(R \setminus t_2)$. But recall that $\chi(R \setminus t_2) = r = \chi(R)$. So, $H$ is a $k$-hyperhole in $R$ of chromatic number $\chi(R)$, which is what we needed. 
\end{proof}

\section{Coloring rings} \label{sec:col}

We remind the reader that $\mathcal{R}_{\geq 4}$ is the class of all graphs $G$ that have the property that every induced subgraph of $G$ either is a ring or has a simplicial vertex. By Lemma~\ref{lemma-RRk-hered}, $\mathcal{R}_{\geq 4}$ is hereditary and contains all rings. Our goal in this section is to construct a polynomial-time coloring algorithm for graphs in $\mathcal{R}_{\geq 4}$ (see Theorem~\ref{thm-ring-col-alg}), and more generally, for graphs in $\mathcal{G}_{\text{T}}$ (see Theorem~\ref{thm-GT-col-alg}). We already know how to color even rings (see Lemma~\ref{lemma-even-ring-col-alg}). In the remainder of the section, we focus primarily on odd rings. 

The following lemma is an easy corollary of Theorem~\ref{thm-ring-hyperhole} and Lemma~\ref{lemma-main-technical}, and it is at the heart of our coloring algorithm for odd rings.

\begin{lemma} \label{lemma-reduce-chi} Let $k \geq 5$ be an odd integer, let $R$ be a $k$-ring with ring partition $(X_1,\dots,X_k)$, and for each $i \in \{1,\dots,k\}$, let $X_i = \{u_i^1,\dots,u_i^{|X_i|}\}$ be an ordering of $X_i$ such that $X_i \subseteq N_R[u_i^{|X_i|}] \subseteq \dots \subseteq N_R[u_i^1] = X_{i-1} \cup X_i \cup X_{i+1}$. For all $i \in \{1,\dots,k\}$, set $s_i = u_i^1$ and $t_i = u_i^{|X_i|}$. Let $c$ be an unimprovable coloring of $R \setminus t_2$, and let $r$ be the number of colors used by $c$.\footnote{In particular, $c$ is a proper coloring of $R \setminus t_2$. Furthermore, we have that $\chi(R \setminus t_2) \leq r$, and this inequality may possibly be strict.} Let $c_1 = c(s_1)$, and let $S = \{x \in V(R) \mid x \neq t_2, c(x) = c_1\}$.\footnote{Note that $S$ is a stable set in $R \setminus t_2$.} Then either $\chi(R \setminus S) \leq r-1$ or $\chi(R) = r+1$. 
\end{lemma} 
\begin{proof} 
By Theorem~\ref{thm-ring-hyperhole}, the hypotheses of Lemma~\ref{lemma-main-technical}(b) are satisfied, and we deduce that either $\chi(R \setminus S) \leq r-1$, or $R$ contains a $k$-hyperhole of chromatic number $r+1$. In the former case, we are done. So assume that $R$ contains a $k$-hyperhole $H$ such that $\chi(H) = r+1$. But then 
\begin{displaymath} 
\begin{array}{ccc cc cc cc} 
r+1 & = & \chi(H) & \leq & \chi(R) & \leq & \chi(R \setminus t_2)+1 & \leq & r+1, 
\end{array} 
\end{displaymath} 
and we deduce that $\chi(R) = r+1$. 
\end{proof}

\begin{lemma} \label{lemma-extend-optimal-coloring} There exists an algorithm with the following specifications: 
\begin{itemize} 
\item Input: All the following: 
\begin{itemize} 
\item an odd ring $R$, 
\item a ring partition $(X_1,\dots,X_k)$ of $R$, 
\item for all $i \in \{1,\dots,k\}$, an ordering $X_i = \{u_i^1,\dots,u_i^{|X_i|}\}$ of $X_i$ such that $X_i \subseteq N_R[u_i^{|X_i|}] \subseteq \dots \subseteq N_R[u_i^1] = X_{i-1} \cup X_i \cup X_{i+1}$, 
\item a proper coloring $c$ of $R \setminus u_2^{|X_2|}$; 
\end{itemize} 
\item Output: A proper coloring of $R$ that uses at most $\max\{\chi(R),r\}$ colors, where $r$ is the number of colors used by the input coloring $c$ of $R \setminus u_2^{|X_2|}$;\footnote{Thus, the algorithm outputs a proper coloring of $R$ that is either optimal or uses no more colors than the input coloring $c$ of $R \setminus u_2^{|X_2|}$ does. Clearly, $\chi(R) \leq \chi(R \setminus t_2)+1 \leq r+1$. So, the output coloring of $R$ uses at most $\max\{\chi(R),r\} \leq r+1$ colors. Furthermore, if it uses exactly $r+1$ colors, then $\chi(R) = r+1$, and our output coloring of $R$ is optimal. However, if our output coloring of $R$ uses at most $r$ colors, then we do not know whether or not the coloring is optimal.} 
\item Running time: $O(n^5)$. 
\end{itemize} 
\end{lemma} 
\begin{proof} 
To simplify notation, for all $i \in \{1,\dots,k\}$, we set $s_i = u_i^1$ and $t_i = u_i^{|X_i|}$. So, $c$ is a proper coloring of $R \setminus t_2$. We may assume that $c$ uses the color set $\{1,\dots,r\}$, i.e.\ that $c:V(R) \setminus \{t_2\} \rightarrow \{1,\dots,r\}$. 

First, we update $c$ by running the algorithm from Lemma~\ref{lemma-unimprov-alg} and transforming it into an unimprovable coloring of $R \setminus t_2$ that uses only colors from the set $\{1,\dots,r\}$; this takes $O(n^4)$ time. We may assume that $c(s_1) = r$. Let $S = \{x \in V(R) \mid x \neq t_2, c(x) = r\}$.\footnote{Clearly, $S$ is a stable set.} Our first goal is to compute a proper coloring $\widetilde{c}$ of $R \setminus S$ that uses at most $\max\{\chi(R \setminus S),r-1\}$ colors. Then, depending on how many colors $\widetilde{c}$ uses, we will construct the needed coloring of $R$ by either extending the coloring $c$ of $R \setminus t_2$ or by extending the coloring $\widetilde{c}$ of $R \setminus S$. 

Let $v_1,\dots,v_t$ ($t \geq 0$) be a maximal sequence of pairwise distinct vertices of $R \setminus S$ such that for all $i \in \{1,\dots,t\}$, $v_i$ is simplicial in $R \setminus (S \cup \{v_1,\dots,v_{i-1}\})$; this sequence can be found in $O(n^3)$ time by running the algorithm from Lemma~\ref{lemma-simplicial-list-alg} with input $R \setminus S$. Suppose first that $V(R) \setminus S = \{v_1,\dots,v_t\}$. Then $v_1,\dots,v_t$ is a simplicial elimination ordering of $R \setminus S$, and we construct the coloring $\widetilde{c}$ by coloring $R \setminus S$ greedily using the ordering $v_t,\dots,v_1$. Clearly, $\widetilde{c}$ uses only $\omega(R \setminus S)$ colors, and we have that $\omega(R \setminus S) \leq \chi(R \setminus S) \leq \max\{\chi(R \setminus S),r-1\}$. 

Suppose now that $V(R) \setminus S \neq \{v_1,\dots,v_t\}$. Set $R' := R \setminus (S \cup \{v_1,\dots,v_t\})$. The maximality of $v_1,\dots,v_t$ guarantees that $R'$ has no simplicial vertices, and so it follows from Lemma~\ref{lemma-ring-simplicial} that $R'$ is a $k$-ring with ring partition $\Big(X_1 \cap V(R'),\dots,X_k \cap V(R')\Big)$. Let $c' = c \upharpoonright (V(R') \setminus \{t_2\})$, and note that $c'$ uses only colors from the set $\{1,\dots,r-1\}$. If $t_2 \notin V(R')$, then we set $c'' := c'$. On the other hand, if $t_2 \in V(R')$, then we make a recursive call to the algorithm with input $R'$ and $c'$,\footnote{We also input the ring partition $\Big(X_1 \cap V(R'),\dots,X_k \cap V(R')\Big)$ of $R'$, and well as the orderings of the sets $X_i \cap V(R')$ inherited from our input orderings of the sets $X_i$.} and we obtain a proper coloring $c''$ of $R'$ that uses at most $\max\{\chi(R'),r-1\}$ colors. So, in either case (i.e.\ independently of whether $t_2$ does or does not belong to $V(R')$), we have now obtained a proper coloring $c''$ of $R'$ that uses at most $\max\{\chi(R'),r-1\}$ colors. We now extend $c''$ to a proper coloring $\widetilde{c}$ of $R \setminus S$ by assigning colors greedily to the vertices $v_t,\dots,v_1$ (in that order). Note that the coloring $\widetilde{c}$ uses at most $\max\{\omega(R \setminus S),\chi(R'),r-1\} \leq \max\{\chi(R \setminus S),r-1\}$ colors. 

In either case,\footnote{That is: both in the case when $V(R) \setminus S = \{v_1,\dots,v_t\}$ and in the case when $V(R) \setminus S \neq \{v_1,\dots,v_t\}$.} we have constructed a proper coloring $\widetilde{c}$ of $R \setminus S$ that uses at most $\max\{\chi(R \setminus S),r-1\}$ colors. If $\widetilde{c}$ uses at most $r-1$ colors, then we extend $\widetilde{c}$ to a proper coloring of $R$ that uses at most $r$ colors by assigning the same new color to all the vertices of the stable set $S$; we then return this coloring of $R$, and we stop. Suppose now that the coloring $\widetilde{c}$ uses at least $r$ colors. Then $\chi(R \setminus S) \geq r$, and so Lemma~\ref{lemma-reduce-chi} implies that $\chi(R) = r+1$. We now extend the coloring $c$ of $R \setminus t_2$ to a proper coloring of $R$ by assigning color $r+1$ to the vertex $t_2$. Our coloring of $R$ uses at most $r+1 = \chi(R)$ colors,\footnote{So, in fact, our coloring of $R$ uses exactly $\chi(R)$ colors, and it is therefore optimal.} we return this coloring, and we stop. 

Clearly, the algorithm is correct. We make $O(n)$ recursive calls, and otherwise, the slowest step of the algorithm takes $O(n^4)$ time. Thus, the total running time of the algorithm is $O(n^5)$. 
\end{proof}

\begin{theorem} \label{thm-ring-col-alg} There exists an algorithm with the following specifications: 
\begin{itemize} 
\item Input: A graph $G$; 
\item Output: Either an optimal coloring of $G$, or the true statement that $G \notin \mathcal{R}_{\geq 4}$; 
\item Running time: $O(n^6)$. 
\end{itemize} 
\end{theorem} 
\begin{proof} 
First, we form a maximal sequence $v_1,\dots,v_t$ ($t \geq 0$) of pairwise distinct vertices of $G$ such that, for all $i \in \{1,\dots,t\}$, $v_i$ is simplicial in $G \setminus \{v_1,\dots,v_{i-1}\}$; this can be done in $O(n^3)$ time by running the algorithm from Lemma~\ref{lemma-simplicial-list-alg} with input $G$. 

Suppose first that $t \geq 1$. If $V(G) = \{v_1,\dots,v_t\}$, so that $v_1,\dots,v_t$ is a simplicial elimination ordering of $G$, then we color $G$ greedily in $O(n^2)$ time using the ordering $v_t,\dots,v_1$; clearly, the resulting coloring of $G$ is optimal, we return this coloring, and we stop. So assume that $V(G) \setminus \{v_1,\dots,v_t\} \neq \emptyset$. We then make a recursive call to the algorithm with input $G \setminus \{v_1,\dots,v_t\}$. If we obtain an optimal coloring of $G \setminus \{v_1,\dots,v_t\}$, then we greedily extend this coloring to an optimal coloring of $G$ using the ordering $v_t,\dots,v_1$, we return this coloring of $G$, and we stop. On the other hand, if the algorithm returns the statement that $G \setminus \{v_1,\dots,v_t\} \notin \mathcal{R}_{\geq 4}$, then we return the answer that $G \notin \mathcal{R}_{\geq 4}$ (this is correct because $R_{\geq 4}$ is hereditary), and we stop. 

From now on, we assume that $t = 0$. Thus, $G$ contains no simplicial vertices, and so by the definition of $\mathcal{R}_{\geq 4}$, either $G$ is a ring, or $G \notin \mathcal{R}_{\geq 4}$. We now run the algorithm from Lemma~\ref{lemma-detect-ring-ord} input $G$; this takes $O(n^2)$ time. If the algorithm returns the answer that $G$ is not a ring, then we return the answer that $G \notin \mathcal{R}_{\geq 4}$. So assume the algorithm returned the statement that $G$ is a ring, along with the length $k$ and ring partition $(X_1,\dots,X_k)$ of $G$, and for each $i \in \{1,\dots,k\}$ an ordering $X_i = \{u_i^1,\dots,u_i^{|X_i|}\}$ of $X_i$ such that $X_i \subseteq N_G[u_i^{|X_i|}] \subseteq \dots \subseteq N_G[u_i^1] = X_{i-1} \cup X_i \cup X_{i+1}$. If $k$ is even, then we obtain an optimal coloring of $G$ in $O(n^3)$ time by running the algorithm from Lemma~\ref{lemma-even-ring-col-alg}, we return this coloring, and we stop. So from now on, we assume that $k$ is odd, so that $G$ is an odd ring. For each $i \in \{1,\dots,k\}$, we set $t_i = u_i^{|X_i|}$. Since $G$ is a ring, Lemma~\ref{lemma-RRk-hered} guarantees that $G \in \mathcal{R}_{\geq 4}$, and since $\mathcal{R}_{\geq 4}$ is hereditary, we see that $G \setminus t_2$ belongs to $\mathcal{R}_{\geq 4}$. We now obtain an optimal coloring $c$ of $G \setminus t_2$ by making a recursive call to the algorithm. We then call the algorithm from Lemma~\ref{lemma-extend-optimal-coloring} with input $G$ and $c$,\footnote{We also input the ring partition $(X_1,\dots,X_k)$ of the ring $G$, as well as our orderings of the sets $X_1,\dots,X_k$.} and we obtain a proper coloring of $G$ that uses at most $\max\{\chi(G),\chi(G \setminus t_2)\} = \chi(G)$ colors;\footnote{Clearly, this coloring of $G$ is optimal.} this takes $O(n^5)$ time. We now return this coloring of $G$, and we stop. 

Clearly, the algorithm is correct. We make $O(n)$ recursive calls to the algorithm, and otherwise, the slowest step of the algorithm takes $O(n^5)$ time. Thus, the total running time of the algorithm is $O(n^6)$. 
\end{proof} 

We complete this section by giving a polynomial-time coloring algorithm for graphs in $\mathcal{G}_{\text{T}}$. 

\begin{theorem} \label{thm-GT-col-alg} There exists an algorithm with the following specifications: 
\begin{itemize} 
\item Input: A graph $G$; 
\item Output: Either an optimal coloring of $G$, or the true statement that $G \notin \mathcal{G}_{\text{T}}$; 
\item Running time: $O(n^7)$. 
\end{itemize} 
\end{theorem} 
\begin{proof} 
We first check whether $G$ has a clique-cutset, and if so, we obtain a clique-cut-partition $(A,B,C)$ of $G$ such that $G[A \cup C]$ does not admit a clique-cutset; this can be done in $O(n^3)$ time by running the algorithm from~\cite{Tarjan} with input $G$. If we obtained the answer that $G$ does not admit a clique-cutset, then we set $A = V(G)$, $B = \emptyset$, and $C = \emptyset$. On the other hand, if we obtained $(A,B,C)$, then we make a recursive call to the algorithm with input $G[B \cup C]$; if we obtained the answer that $G[B \cup C] \notin \mathcal{G}_{\text{T}}$, then we return the answer that $G \notin \mathcal{G}_{\text{T}}$ (this is correct because $\mathcal{G}_{\text{T}}$ is hereditary), and we stop. So from now on, we assume that one of the following holds: 
\begin{itemize} 
\item $B = C = \emptyset$; 
\item $(A,B,C)$ is a clique-cut-partition of $G$, and we recursively obtained an optimal coloring $c_B$ of $G[B \cup C]$. 
\end{itemize} 
In either case, we also have that $G[A \cup C]$ does not admit a clique-cutset. 

We now run the algorithm from Theorem~\ref{thm-ring-col-alg} with input $G[A \cup C]$; this takes $O(n^6)$ time. The algorithm either returns an optimal coloring $c_A$ of $G[A \cup C]$, or it returns the answer that $G[A \cup C] \notin \mathcal{R}_{\geq 4}$. If the algorithm returned the answer that $G[A \cup C] \notin \mathcal{R}_{\geq 4}$, then our goal is to either produce an optimal coloring $c_A$ of $G[A \cup C]$ in another way, or to determine that $G \notin \mathcal{G}_{\text{T}}$. In this case (i.e.\ if the algorithm returned the answer that $G[A \cup C] \notin \mathcal{R}_{\geq 4}$), we proceed as follows. Since $\mathcal{R}_{\geq 4}$ contains all rings (by Lemma~\ref{lemma-RRk-hered}), we have that $G[A \cup C]$ is not a ring. Recall that $G[A \cup C]$ does not admit a clique-cutset. Thus, Theorem~\ref{thm-GT-decomp} implies that either $G[A \cup C]$ is a complete graph, or $G[A \cup C]$ is a 7-hyperantihole, or $G[A \cup C] \notin \mathcal{G}_{\text{T}}$ (in which case, $G \notin \mathcal{G}_{\text{T}}$, since $\mathcal{G}_{\text{T}}$ is hereditary). Clearly, complete graphs have stability number one, and hyperantiholes have stability number two. Thus, either $\alpha(G[A \cup C]) \leq 2$ or $G \notin \mathcal{G}_{\text{T}}$. We determine whether $\alpha(G[A \cup C]) \leq 2$ by examining all triples of vertices in $G[A \cup C]$; this takes $O(n^3)$ time. If $\alpha(G[A \cup C]) \geq 3$, then we return the answer that $G \notin \mathcal{G}_{\text{T}}$, and we stop. So suppose that $\alpha(G[A \cup C]) \leq 2$. This means that each color class of a proper coloring of $G[A \cup C]$ is of size at most two, and that, taken together, color classes of size exactly two correspond to a matching of $\overline{G}[A \cup C]$ (the complement of $G[A \cup C]$). So, we form the graph $\overline{G}[A \cup C]$ in $O(n^2)$ time, and we find a maximum matching $M$ in $\overline{G}[A \cup C]$ in $O(n^4)$ time by running the algorithm from~\cite{matching}. We now color $G[A \cup C]$ as follows: each member of $M$ is a two-vertex color class,\footnote{By construction, members of $M$ are edges of $\overline{G}[A \cup C]$; consequently, members of $M$ are stable sets of size two in $G[A \cup C]$.} and each vertex in $A \cup C$ that is not an endpoint of any member of $M$ forms a one-vertex color class.\footnote{So, in total, we used $|M|+(|A \cup C|-2|M|) = |A \cup C|-|M|$ colors.} This produces an optimal coloring $c_A$ of $G[A \cup C]$. 

So from now on, we may assume that we have obtained an optimal coloring $c_A$ of $G[A \cup C]$. If $B = C = \emptyset$, then $c_A$ is in fact an optimal coloring of $G$; in this case, we return $c_A$, and we stop. So assume that $B \cup C \neq \emptyset$. Then we have already obtained an optimal coloring $c_B$ of $G[B \cup C]$. After possibly renaming colors, we may assume that the color set used by one of $c_A,c_B$ is included in the color set used by the other one. Now, $C$ is a clique in $G$, and so $c_A$ assigns a different color to each vertex of $C$, and the same is true for $c_B$. So, after possibly permuting colors, we may assume that $c_A$ and $c_B$ agree on $C$, i.e.\ that $c_A \upharpoonright C = c_B \upharpoonright C$. Now $c := c_A \cup c_B$ is an optimal coloring of $G$. We return $c$, and we stop. 

Clearly, the algorithm is correct. We make $O(n)$ recursive calls to the algorithm, and otherwise, the slowest step takes $O(n^6)$ time. Thus, the total running time of the algorithm is $O(n^7)$. 
\end{proof}

\section{Computing the chromatic number of a ring} \label{sec:chi}

In section~\ref{sec:col}, we gave an $O(n^6)$ time coloring algorithm for graphs in $\mathcal{R}_{\geq 4}$ (see Theorem~\ref{thm-ring-col-alg}),\footnote{Recall that, by Lemma~\ref{lemma-RRk-hered}, the class $\mathcal{R}_{\geq 4}$ contains all rings.} and we also gave an $O(n^7)$ time coloring algorithm for graphs in $\mathcal{G}_{\text{T}}$ (see Theorem~\ref{thm-GT-col-alg}). These algorithms produce optimal colorings of input graphs from the specified classes; however, for some applications, it is enough to compute the chromatic number, without actually finding an optimal coloring of the input graph. In this section, we use Corollary~\ref{cor-ring-chi-formula} and Lemma~\ref{lemma-omega-ring} to construct an $O(n^3)$ time algorithm that computes the chromatic number of a graph in $\mathcal{R}_{\geq 4}$ (see Theorem~\ref{thm-ring-chi-alg}), and using this result, we construct an $O(n^5)$ time algorithm that computes the chromatic number of a graph in $\mathcal{G}_{\text{T}}$ (see Theorem~\ref{thm-GT-chi-alg}). 

First, we give an $O(n^3)$ time algorithm that computes a maximum hyperhole in a ring (see Lemma~\ref{lemma-normal-hyperhole-alg}).\footnote{We remind the reader that, by Lemma~\ref{lemma-ring-chordal}(b), every hyperhole in a ring is of the same length as that ring. As usual, a {\em maximum hyperhole} in a ring is a hyperhole of maximum size in that ring.} We begin with some terminology and notation. Let $k \geq 4$ be an integer, let $R$ be a $k$-ring with ring partition $(X_1,\dots,X_k)$, and for all $i \in \{1,\dots,k\}$, let $X_i = \{u_i^1,\dots,u_i^{|X_i|}\}$ be an ordering of $X_i$ such that $X_i \subseteq N_G[u_i^{|X_i|}] \subseteq \dots \subseteq N_G[u_i^1] = X_{i-1} \cup X_i \cup X_{i+1}$, as in the definition of a ring. Let $H$ be a hyperhole in $R$. By Lemma~\ref{lemma-ring-chordal}, the hyperhole $H$ is of length $k$, and it intersects each of the sets $X_1,\dots,X_k$. For all $i \in \{1,\dots,k\}$, let $\ell_i = \max\{\ell \in \{1,\dots,|X_i|\} \mid u_i^{\ell} \in V(H)\}$ and $Y_i = \{u_i^1,\dots,u_i^{\ell_i}\}$. Finally, let $\widetilde{H} = R[Y_1 \cup \dots \cup Y_k]$ and $C_H = \{u_1^{\ell_1},\dots,u_k^{\ell_k}\}$. Clearly, $\widetilde{H}$ is a hyperhole, with $V(H)\subseteq V(\widetilde{H})$. Furthermore, $C_H$ induces a hole in $R$, and it uniquely determines $\widetilde{H}$. We say that $H$ is \emph{normal} in $R$ if $H = \widetilde{H}$. Clearly, any maximal hyperhole (and therefore, any hyperhole of maximum size) in $R$ is normal. Thus, to find a maximum hyperhole in an input ring, we need only consider normal hyperholes in that ring.

\begin{lemma} \label{lemma-normal-hyperhole-alg} There exists an algorithm with the following specifications: 
\begin{itemize} 
\item Input: A graph $R$; 
\item Output: Either a maximum hyperhole $H$ in $R$, or the true statement that $R$ is not a ring; 
\item Running time: $O(n^3)$. 
\end{itemize} 
\end{lemma} 
\begin{proof} 
We first run the algorithm from Lemma~\ref{lemma-detect-ring-ord} with input $R$; this takes $O(n^2)$ time. If the algorithm returns the answer that $R$ is not a ring, then we return that answer as well and stop. So assume the algorithm returned the statement that $R$ is a ring, along with the length $k$ and ring partition $(X_1,\dots,X_k)$ of $R$, and for each $i \in \{1,\dots,k\}$ an ordering $X_i = \{u_i^1,\dots,u_i^{|X_i|}\}$ of $X_i$ such that $X_i \subseteq N_R[u_i^{|X_i|}] \subseteq \dots \subseteq N_R[u_i^1] = X_{i-1} \cup X_i \cup X_{i+1}$. 

For each $j \in \{1,\dots,|X_1|\}$, we will find a normal hyperhole $H_j$ of $R$ such that $V(H_j) \cap X_1 = \{u_1^1,\dots,u_1^j\}$, and subject to that, $|V(H_j)|$ is maximum. We will then compare the sizes of all the $H_j$'s (with $1 \leq j \leq |X_1|$), and we will return an $H_j$ of maximum size. 

We begin by constructing an auxiliary weighted digraph $(D,w)$,\footnote{The weight function $w$ will assign nonnegative interger weights to the arcs of $D$.} as follows. First, we construct a set of $|X_1|$ new vertices, $X_{k+1} = \{u_{k+1}^1,\dots,u_{k+1}^{|X_1|}\}$, with $X_{k+1} \cap V(R) = \emptyset$.\footnote{So, $|X_{k+1}| = |X_1|$, and we think of $X_{k+1}$ as a copy of $X_1$.} Let ${D}$ be the digraph with vertex set $V({D}) = V(R) \cup X_{k+1}$ and arc set:
\begin{displaymath} 
\begin{array}{rcl} 
A({D}) & = & \bigcup\limits_{i=1}^{k-1} \Big(\Big\{\overrightarrow{x y} \mid x\in X_i, \ y\in X_{i+1}, \ xy \in E(R)\Big\} 
\\
& & \phantom{iiiii} \cup \Big\{\overrightarrow{x u_{k+1}^{\ell}} \mid x \in X_k, \ x u_1^{\ell} \in
E(R), \ 1 \leq \ell \leq |X_1|\Big\}\Big). 
\end{array} 
\end{displaymath} 
Finally, for every arc $\overrightarrow{u_i^pu_{i+1}^q}$ in $A({D})$, with $i \in \{1, \dots, k\}$, $p \in \{1, \dots, |X_i|\}$, and $q \in \{1, \dots, |X_{i+1}|\}$, we set $w(\overrightarrow{u_i^pu_{i+1}^q}) = (|X_i|-p)+(|X_{i+1}|-q)$.

Now, for a fixed $j \in \{1,\dots,|X_1|\}$, we find the hyperhole $H_j$ as follows. Let ${P_j}$ be a minimum weight directed path between $u_1^j$ and $u_{k+1}^j$ in the weighted digraph $({D}, w)$.  Such a path can be found in $O(n^2)$ time using Dijkstra's algorithm~\cite{D59, Sch}. For each $i \in \{1, \dots, k\}$, let $\ell_{i, j} \in \{1, \dots, |X_i|\}$ be the (unique) index such that $u_i^{\ell_{i, j}} \in V({P_j})$, and let $Y_{i, j} = \{u_i^{\ell} \mid 1 \leq \ell \leq \ell_{i, j}\}$. Finally, let $H_j = R[Y_{1, j}\cup\dots\cup Y_{k, j}]$.  Clearly, $H_j$ is a normal hyperhole of $R$, and $V(H_j) \cap X_1 = \{u_1^1,\dots,u_1^j\}$. Moreover, we have that 
\begin{displaymath} 
\begin{array}{rcl} 
|V(H_j)| & = & \sum\limits_{i=1}^k |Y_{i, j}| 
\\
\\
& = & \sum\limits_{i=1}^k \ell_{i, j} 
\\
\\
& = & |V(R)|-\sum\limits_{i=1}^k (|X_i|-\ell_{i, j}) 
\\
\\
& = & |V(R)|-\frac{1}{2}w({P_j}), 
\end{array} 
\end{displaymath} 
and so the fact that $P_j$ has minimum weight implies that $H_j$ has maximum size among all hyperholes $H$ in $R$ that satisfy $V(H) \cap X_1 = \{u_1^1,\dots,u_1^j\}$. So, $H_j$ is the desired hyperhole for a given $j$. 

We now compare the sizes of the hyperholes $H_1,\dots,H_{|X_1|}$ (this takes $O(n^2)$ time), and we return one of maximum size. 

Clearly, the algorithm is correct. The total running time is $O(n^3)$, since computing $H_j$ (for fixed $j$) takes $O(n^2)$ time, and we do this for $O(n)$ values of $j$. 
\end{proof} 

We now give a polynomial-time algorithm that computes the chromatic number of graphs in $\mathcal{R}_{\geq 4}$. 

\begin{theorem} \label{thm-ring-chi-alg} There exists an algorithm with the following specifications: 
\begin{itemize} 
\item Input: A graph $G$; 
\item Output: Either $\chi(G)$, or the true statement that $G \notin \mathcal{R}_{\geq 4}$; 
\item Running time: $O(n^3)$. 
\end{itemize} 
\end{theorem}  
\begin{proof} 
First, we form a maximal sequence $v_1,\dots,v_t$ ($t \geq 0$) of pairwise distinct vertices of $G$ such that, for all $i \in \{1,\dots,t\}$, $v_i$ is simplicial in $G \setminus \{v_1,\dots,v_{i-1}\}$; this can be done in $O(n^3)$ time by calling the algorithm from Lemma~\ref{lemma-simplicial-list-alg} with input $G$. 

Suppose first that $V(G) = \{v_1,\dots,v_t\}$, so that $v_1,\dots,v_t$ is a simplicial elimination ordering of $G$. In this case, we greedily color $G$ using the ordering $v_t,\dots,v_1$,\footnote{Clearly, this produces an optimal coloring of $G$.} we return the number of colors that we used, and we stop; this takes $O(n^2)$ time. 

From now on, we assume that $V(G) \neq \{v_1,\dots,v_t\}$, and we form the graph $R := G \setminus \{v_1,\dots,v_t\}$ in $O(n^2)$ time. The maximality of $v_1,\dots,v_t$ guarantees that $R$ contains no simplicial vertices, and so by the definition of $\mathcal{R}_{\geq 4}$, we have that either $R$ is a ring, or $G \notin \mathcal{R}_{\geq 4}$.\footnote{Indeed, suppose that $G \in \mathcal{R}_{\geq 4}$. Since $\mathcal{R}_{\geq 4}$ is hereditary, it follows that $R \in \mathcal{R}_{\geq 4}$. By the maximality of $v_1,\dots,v_t$, the graph $R$ has no simplicial vertices. So, by the definition of $\mathcal{R}_{\geq 4}$, we have that $R$ is a ring.} 

We now run the algorithm from Lemma~\ref{lemma-detect-ring} with input $R$; this takes $O(n^2)$ time. If the algorithm returns the answer that $R$ is not a ring, then we return the answer that $G \notin \mathcal{R}_{\geq 4}$, and we stop. So assume the algorithm returned the statement that $R$ is a ring, along with the length $k$ and ring partition $(X_1,\dots,X_k)$ of $R$. Next, we call the algorithm from Lemma~\ref{lemma-omega-ring}; this takes $O(n^3)$ time. Since $R$ is a ring, Lemma~\ref{lemma-ring-chordal}(d) guarantees that $R \in \mathcal{G}_{\text{T}}$, and so the algorithm returns $\omega(R)$. Next, we run the algorithm from Lemma~\ref{lemma-normal-hyperhole-alg} with input $R$; this takes $O(n^3)$ time. Since $R$ is a ring, we know that the algorithm returns a hyperhole $H$ of $R$ of maximum size; since $R$ is a $k$-ring, Lemma~\ref{lemma-ring-chordal}(b) guarantees that $H$ is a $k$-hyperhole. Set $r := \max\Big\{\omega(R),\Big\lceil \frac{|V(H)|}{\lfloor k/2 \rfloor} \Big\rceil\Big\}$; by Corollary~\ref{cor-ring-chi-formula}, we have that $\chi(R) = r$. 

If $t = 0$ (so that $G = R$), then we return $r$, and we stop. So assume that $t \geq 1$. For each $i \in \{1,\dots,t\}$, set $r_i = |N_G[v_i] \setminus \{v_1,\dots,v_{i-1}\}|$; computing the constants $r_1,\dots,r_t$ takes $O(n^2)$ time. An easy induction using Lemma~\ref{lemma-simplicial-chi} now establishes that $\chi(G) = \max\{r_1,\dots,r_t,r\}$. So, we return $\max\{r_1,\dots,r_t,r\}$, and we stop. 

Clearly, the algorithm is correct, and its running time is $O(n^3)$. 
\end{proof} 

We complete this section by showing how to compute the chromatic number of graphs in $\mathcal{G}_{\text{T}}$ in polynomial time. We remark that this algorithm is very similar to the one from Theorem~\ref{thm-GT-col-alg}, except that we use Theorem~\ref{thm-ring-chi-alg} instead of Theorem~\ref{thm-ring-col-alg}. Nevertheless, for the sake of completeness, we give all the details. 

\begin{theorem} \label{thm-GT-chi-alg} There exists an algorithm with the following specifications: 
\begin{itemize} 
\item Input: A graph $G$; 
\item Output: Either $\chi(G)$, or the true statement that $G \notin \mathcal{G}_{\text{T}}$; 
\item Running time: $O(n^5)$. 
\end{itemize} 
\end{theorem} 
\begin{proof} 
We first check whether $G$ has a clique-cutset, and if so, we obtain a clique-cut-partition $(A,B,C)$ of $G$ such that $G[A \cup C]$ does not admit a clique-cutset; this can be done by running the algorithm from~\cite{Tarjan} with input $G$, and it takes $O(n^3)$ time. If we obtained the answer that $G$ does not admit a clique-cutset, then we set $A = V(G)$, $B = \emptyset$, and $C = \emptyset$, and we set $r = 0$. On the other hand, if we obtained $(A,B,C)$, then we make a recursive call to the algorithm with input $G[B \cup C]$; if we obtained the answer that $G[B \cup C] \notin \mathcal{G}_{\text{T}}$, then we return the answer that $G \notin \mathcal{G}_{\text{T}}$ and stop,\footnote{This is correct because $\mathcal{G}_{\text{T}}$ is hereditary.} and otherwise (i.e.\ if we obtained the chromatic number of $G[B \cup C]$) we set $r = \chi(G[B \cup C])$. 

We may now assume that we have obtained the number $r$ (for otherwise, we terminated the algorithm). Clearly, $\chi(G) = \max\{\chi(G[A \cup C]),r\}$. Next, we run the algorithm from Theorem~\ref{thm-ring-chi-alg} with input $G[A \cup C]$; this takes $O(n^3)$ time. If the algorithm returned $\chi(G[A \cup C])$, then we return the number $\max\{\chi(G[A \cup C]),r\}$, and we stop. So assume the algorithm returned the answer that $G[A \cup C]$ is not a ring. 

So far, we know that $G[A \cup C]$ does not admit a clique-cutset and is not a ring. Theorem~\ref{thm-GT-decomp} now guarantees that either $G[A \cup C]$ is a complete graph or a 7-hyperantihole, or $G[A \cup C] \notin \mathcal{G}_{\text{T}}$ (in which case, $G \notin \mathcal{G}_{\text{T}}$, since $\mathcal{G}_{\text{T}}$ is hereditary). Clearly, complete graphs have stability number one, and hyperantiholes have stability number two. Thus, either $\alpha(G[A \cup C]) \leq 2$ or $G \notin \mathcal{G}_{\text{T}}$. Now, we determine whether $\alpha(G[A \cup C]) \leq 2$ by examining all triples of vertices in $G[A \cup C]$; this takes $O(n^3)$ time. If $\alpha(G[A \cup C]) \geq 3$, then we return the answer that $G \notin \mathcal{G}_{\text{T}}$ and stop. Assume now that $\alpha(G[A \cup C]) \leq 2$. Then we form the graph $\overline{G}[A \cup C]$ (the complement of $G[A \cup C]$) in $O(n^2)$ time, and we find a maximum matching $M$ in $\overline{G}[A \cup C]$ by running the algorithm from~\cite{matching}; this takes $O(n^4)$ time. Since $\alpha(G[A \cup C]) \leq 2$, we see that $\chi(G[A \cup C]) = |A \cup C|-|M|$; we now return the number $\max\{|A \cup C|-|M|,r\}$, and we stop. 

Clearly, the algorithm is correct. The slowest step takes $O(n^4)$ time, and we make $O(n)$ recursive calls. Thus, the total running time of the algorithm is $O(n^5)$. 
\end{proof}

\section{Optimal $\chi$-bounding functions} \label{sec:chibound} 

For all integers $k \geq 4$, we let $\mathcal{H}_k$ be the class of all induced subgraphs of $k$-hyperholes, and we let $\mathcal{A}_k$ be the class of all induced subgraphs of $k$-hyperantiholes; clearly, classes $\mathcal{H}_k$ and $\mathcal{A}_k$ are both hereditary, and they contain all complete graphs.\footnote{The reason we emphasize that these classes contain all complete graphs is that we defined optimal $\chi$-bounding functions only for hereditary, $\chi$-bounded classes that contain all complete graphs.} Recall that for all integers $k \geq 4$, $\mathcal{R}_k$ is the class of all graphs $G$ that have the property that every induced subgraph of $G$ either is a $k$-ring or has a simplicial vertex; clearly, $\mathcal{R}_k$ is hereditary and contains all complete graphs, and by Lemma~\ref{lemma-RRk-hered}, all $k$-rings belong to $\mathcal{R}_k$ (in particular, $\mathcal{H}_k \subseteq \mathcal{R}_k$). In this section, for all integers $k \geq 4$, we find the optimal $\chi$-bounding functions for the classes $\mathcal{H}_k$ (see Theorem~\ref{thm-HHk-chi-optimal}), $\mathcal{A}_k$ (see Theorem~\ref{thm-AAk-chi-optimal}), and $\mathcal{R}_k$ (see Theorem~\ref{thm-RRk-chi-optimal}). Further, for all integers $k \geq 4$, we set $\mathcal{H}_{\geq k} = \bigcup\limits_{i=k}^{\infty} \mathcal{H}_i$ and $\mathcal{A}_{\geq k} = \bigcup\limits_{i=k}^{\infty} \mathcal{A}_i$, and we remind the reader that $\mathcal{R}_{\geq k} = \bigcup\limits_{i=k}^{\infty} \mathcal{R}_i$.\footnote{Clearly, for all integers $k \geq 4$ we have that: $\mathcal{H}_{\geq k}$ is the class of all induced subgraphs of hyperholes of length at least $k$; $\mathcal{A}_{\geq k}$ is the class of all induced subgraphs of hyperantiholes of length at least $k$; and $\mathcal{R}_{\geq k}$ contains all induced subgraphs of rings of length at least $k$. In particular, $\mathcal{H}_{\geq k} \subseteq \mathcal{R}_{\geq k}$. It is clear that $\mathcal{H}_{\geq k}$, $\mathcal{A}_{\geq k}$, and $\mathcal{R}_{\geq k}$ are hereditary and contain all complete graphs.} For all integers $k \geq 4$, we find the optimal $\chi$-bounding functions for the classes $\mathcal{H}_{\geq k}$ (see Corollary~\ref{cor-HHk-chi-optimal}), $\mathcal{A}_{\geq k}$ (see Corollary~\ref{cor-AAk-chi-optimal}), and $\mathcal{R}_{\geq k}$ (see Corollary~\ref{cor-RRk-chi-optimal}); see also Theorem~\ref{thm-HH-AA-RR-chi-optimal}. Finally, we find the optimal $\chi$-bounding function for the class $\mathcal{G}_{\text{T}}$ (see Theorem~\ref{thm-GT-chi}). 

Recall that $\mathbb{N}$ is the set of all positive integers, and let $i_\mathbb{N}$ be the identity function on $\mathbb{N}$, i.e.\ let $i_{\mathbb{N}}:\mathbb{N} \rightarrow \mathbb{N}$ be given by $i_{\mathbb{N}}(n) = n$ for all $n \in \mathbb{N}$. 

We define the function $f_{\text{T}}:\mathbb{N} \rightarrow \mathbb{N}$ by setting 
\begin{displaymath} 
\begin{array}{lll} 
f_{\text{T}}(n) & = & \left\{\begin{array}{lll} 
\lfloor 5n/4 \rfloor & \text{if} & \text{$n \equiv 0,1$ (mod $4$)} 
\\
\\
\lceil 5n/4 \rceil & \text{if} & \text{$n \equiv 2,3$ (mod $4$)} 
\end{array}\right. 
\end{array}
\end{displaymath} 
for all $n \in \mathbb{N}$.

For all odd integers $k \geq 5$, we define the function $f_k:\mathbb{N} \rightarrow \mathbb{N}$ by setting 
\begin{displaymath} 
\begin{array}{lll} 
f_k(n) & = & \left\{\begin{array}{lll} 
\Big\lfloor \frac{kn}{k-1} \Big\rfloor & \text{if} & \text{$n \equiv 0,1$ (mod $k-1$)} 
\\
\\
\Big\lceil \frac{kn}{k-1} \Big\rceil & \text{if} & \text{$n \equiv 2,\dots,k-2$ (mod $k-1$)}
\end{array}\right. 
\end{array} 
\end{displaymath} 
for all $n \in \mathbb{N}$. 

For all odd integers $k \geq 5$, we define the function $g_k:\mathbb{N} \rightarrow \mathbb{N}$ by setting 
\begin{displaymath} 
\begin{array}{lll} 
g_k(n) & = & \left\{\begin{array}{lll} 
\Big\lfloor \frac{kn}{k-1} \Big\rfloor & \text{if} & \text{$n \equiv 0,\dots,\frac{k-3}{2}$ (mod $k-1$)} 
\\
\\
\Big\lceil \frac{kn}{k-1} \Big\rceil & \text{if} & \text{$n \equiv \frac{k-1}{2},\dots,k-2$ (mod $k-1$)}
\end{array}\right. 
\end{array} 
\end{displaymath} 
for all $n \in \mathbb{N}$. 

Note that $f_{\text{T}} = f_5 = g_5$. Before turning to the classes mentioned at the beginning of this section, we prove a few technical lemmas concerning functions $f_{\text{T}}$, $f_k$, and $g_k$. 

\begin{lemma} \label{lemma-fk-calculation} Let $k \geq 5$ be an odd integer, and let $n \in \mathbb{N}$. Then $f_k(n) = n+\Big\lceil \frac{2 \lfloor n/2 \rfloor}{k-1} \Big\rceil$. 
\end{lemma} 
\begin{proof} 
Set $m = \lfloor \frac{n}{k-1} \rfloor$ and $\ell = n-(k-1)m$. Clearly, $m$ is a nonnegative integer, $\ell \in \{0,\dots,k-2\}$, $n = (k-1)m+\ell$, and $n \equiv \ell$ (mod $k-1$). 

Since $k$ is odd, we have that $k-1$ is even, and so 
\begin{displaymath} 
\begin{array}{ccc cc} 
\Big\lceil \frac{2 \lfloor n/2 \rfloor}{k-1} \Big\rceil & = & \Bigg\lceil \frac{2 \Big\lfloor \frac{(k-1)m+\ell}{2} \Big\rfloor}{k-1} \Bigg\rceil & = & m+\Big\lceil \frac{2\lfloor \ell/2 \rfloor}{k-1} \Big\rceil. 
\end{array} 
\end{displaymath} 

If $0 \leq \ell \leq 1$, then $f_k(n) = \lfloor \frac{kn}{k-1} \rfloor$, and we have that 
\begin{displaymath}
\begin{array}{rcl} 
n+\Big\lceil \frac{2 \lfloor n/2 \rfloor}{k-1} \Big\rceil & = & n+m+\Big\lceil \frac{2\lfloor \ell/2 \rfloor}{k-1} \Big\rceil 
\\
\\
& = & n+m
\\
\\
& = & \lfloor \frac{kn}{k-1} \rfloor 
\\
\\
& = & f_k(n), 
\end{array} 
\end{displaymath} 
and we are done. 

Suppose now that $2 \leq \ell \leq k-2$; then $f_k(n) = \lceil \frac{kn}{k-1} \rceil$. First, we have that 
\begin{displaymath} 
\begin{array}{ccc cccc} 
n+\Big\lceil \frac{2 \lfloor n/2 \rfloor}{k-1} \Big\rceil & = & n+m+\Big\lceil \frac{2\lfloor \ell/2 \rfloor}{k-1} \Big\rceil & = & n+m+1 & = & \lfloor \frac{kn}{k-1} \rfloor+1. 
\end{array} 
\end{displaymath} 
Since $\ell \neq 0$, we see that $\frac{kn}{k-1}$ is not an integer, and so $\lfloor \frac{kn}{k-1} \rfloor+1 = \lceil \frac{kn}{k-1} \rceil$. It now follows that 
\begin{displaymath} 
\begin{array}{ccc cccc} 
n+\Big\lceil \frac{2 \lfloor n/2 \rfloor}{k-1} \Big\rceil & = & \lfloor \frac{kn}{k-1} \rfloor+1 & = & \lceil \frac{kn}{k-1} \rceil & = & f_k(n), 
\end{array} 
\end{displaymath} 
which is what we needed. This completes the argument. 
\end{proof}

\begin{lemma} \label{lemma-gk-calculation} Let $k \geq 5$ be an odd integer, and let $n \in \mathbb{N}$. Then $g_k(n) = n+\Big\lceil \lfloor \frac{2n}{k-1} \rfloor/2 \Big\rceil$. 
\end{lemma} 
\begin{proof} 
Set $m = \lfloor \frac{n}{k-1} \rfloor$ and $\ell = n-(k-1)m$. Clearly, $m$ is a nonnegative integer, $\ell \in \{0,\dots,k-2\}$, $n = (k-1)m+\ell$, and $n \equiv \ell$ (mod $k-1$). 

First, we have that 
\begin{displaymath} 
\begin{array}{ccc cc} 
\Big\lceil \lfloor \frac{2n}{k-1} \rfloor/2 \Big\rceil & = & \Big\lceil \Big\lfloor \frac{2((k-1)m+\ell)}{k-1} \Big\rfloor/2 \Big\rceil & = & m+\Big\lceil \lfloor \frac{2\ell}{k-1} \rfloor/2 \Big\rceil. 
\end{array} 
\end{displaymath} 

Suppose first that $0 \leq \ell \leq \frac{k-3}{2}$; then $g_k(n) = \lfloor \frac{kn}{k-1} \rfloor$. We now have that 
\begin{displaymath} 
\begin{array}{rcl} 
n+\Big\lceil \lfloor \frac{2n}{k-1} \rfloor/2 \Big\rceil & = & n+m+\Big\lceil \lfloor \frac{2\ell}{k-1} \rfloor/2 \Big\rceil 
\\
\\
& = & n+m 
\\
\\
& = & \lfloor \frac{kn}{k-1} \rfloor 
\\
\\
& = & g_k(n), 
\end{array} 
\end{displaymath} 
which is what we needed. 

Suppose now that $\frac{k-1}{2} \leq \ell \leq k-2$; then $g_k(n) = \lceil \frac{kn}{k-1} \rceil$. Now, note that 
\begin{displaymath} 
\begin{array}{rcl} 
n+\Big\lceil \lfloor \frac{2n}{k-1} \rfloor/2 \Big\rceil & = & n+m+\Big\lceil \lfloor \frac{2\ell}{k-1} \rfloor/2 \Big\rceil 
\\
\\
& = & n+m+1 
\\
\\
& = & \lfloor \frac{kn}{k-1} \rfloor+1. 
\end{array} 
\end{displaymath} 
Since $\ell \neq 0$, we see that $\frac{kn}{k-1}$ is not an integer, and so $\lfloor \frac{kn}{k-1} \rfloor+1 = \lceil \frac{kn}{k-1} \rceil$. We now have that 
\begin{displaymath} 
\begin{array}{ccc cccc} 
n+\Big\lceil \lfloor \frac{2n}{k-1} \rfloor/2 \Big\rceil & = & \lfloor \frac{kn}{k-1} \rfloor+1 & = & \lceil \frac{kn}{k-1} \rceil & = & g_k(n), 
\end{array} 
\end{displaymath}
which is what we needed. This completes the argument. 
\end{proof} 

Given functions $f,g:\mathbb{N} \rightarrow \mathbb{N}$, we write $f \leq g$ and $g \geq f$, if for all $n \in \mathbb{N}$, we have that $f(n) \leq g(n)$. As usual, a function $f:\mathbb{N} \rightarrow \mathbb{N}$ is said to be {\em nondecreasing} if for all $n_1,n_2 \in \mathbb{N}$ such that $n_1 \leq n_2$, we have that $f(n_1) \leq f(n_2)$. 

\begin{lemma} \label{lemma-fTfkgk} Function $f_{\text{T}}$ is nondecreasing, and $f_{\text{T}} = f_5 = g_5$. Furthermore, for all odd integers $k \geq 5$, all the following hold: 
\begin{itemize} 
\item[(a)] $f_{\text{T}} \geq f_k \geq g_k$; 
\item[(b)] functions $f_k$ and $g_k$ are nondecreasing; 
\item[(c)] $f_k \geq f_{k+2}$ and $g_k \geq g_{k+2}$. 
\end{itemize} 
\end{lemma} 
\begin{proof} 
The fact that $f_{\text{T}}$ is nondecreasing, and that $f_{\text{T}} = f_5 = g_5$, follows from the definitions of $f_{\text{T}}$, $f_5$, and $g_5$. Further, it follows from construction that for all odd integers $k \geq 5$, we have that $f_k \geq g_k$. The rest readily follows from Lemmas~\ref{lemma-fk-calculation} and~\ref{lemma-gk-calculation}. 
\end{proof} 

\begin{lemma} \label{lemma-hyp-chi} Let $k \geq 5$ be an odd integer. Then all $k$-hyperholes $H$ satisfy $\chi(H) \leq f_k(\omega(H))$. Furthermore, there exists a sequence $\{H_n^k\}_{n=2}^{\infty}$ of $k$-hyperholes such that for all integers $n \geq 2$, we have that $\omega(H_n^k) = n$ and $\chi(H_n^k) = f_k(n)$. 
\end{lemma} 
\begin{proof} 
We begin by proving the first statement of the lemma. Let $H$ be a $k$-hyperhole, and let $(X_1,\dots,X_k)$ be a partition of $V(H)$ into nonempty cliques such that for all $i \in \{1,\dots,k\}$, $X_i$ is complete to $X_{i-1} \cup X_{i+1}$ and anticomplete to $V(H) \setminus (X_{i-1} \cup X_i \cup X_{i+1})$, as in the definition of a $k$-hyperhole. Since $H$ is a $k$-hyperhole, and since $k$ is odd, we have that $\alpha(H) = \lfloor k/2 \rfloor = \frac{k-1}{2}$. Then by Lemma~\ref{lemma-hyperhole-chi-formula}, we have that 
\begin{displaymath} 
\begin{array}{ccccc} 
\chi(H) & = & \max\Big\{\omega(H),\Big\lceil \frac{|V(H)|}{\alpha(H)} \Big\rceil\Big\} & = & \max\Big\{\omega(H),\Big\lceil \frac{2|V(H)|}{k-1} \Big\rceil\Big\}. 
\end{array} 
\end{displaymath} 
It is clear that $\omega(H) \leq f_k(\omega(H))$, and so it suffices to show that $\Big\lceil \frac{2|V(H)|}{k-1} \Big\rceil \leq f_k(\omega(H))$. Clearly, for all $i \in \{1,\dots,k\}$, $X_i \cup X_{i+1}$ is a clique, and so $|X_i|+|X_{i+1}| \leq \omega(H)$. In particular, $|X_k|+|X_1| \leq \omega(H)$, and so either $|X_k| \leq \lfloor \omega(H)/2 \rfloor$ or $|X_1| \leq \lfloor \omega(H)/2 \rfloor$; by symmetry, we may assume that $|X_k| \leq \lfloor \omega(H)/2 \rfloor$. Now, using the fact that $k$ is odd, we get that 
\begin{displaymath} 
\begin{array}{rcl} 
|V(H)| & = & \sum\limits_{i=1}^k |X_i| 
\\
\\
& = & \Big(\sum\limits_{i=1}^{(k-1)/2} (|X_{2i-1}|+|X_{2i}|)\Big)+|X_k| 
\\
\\
& \leq & \frac{k-1}{2}\omega(H)+\lfloor \omega(H)/2 \rfloor. 
\end{array} 
\end{displaymath} 
But now by Lemma~\ref{lemma-fk-calculation}, we have that 
\begin{displaymath} 
\begin{array}{rcl} 
\Big\lceil \frac{2|V(H)|}{k-1} \Big\rceil & \leq & \Bigg\lceil \frac{2\Big(\frac{k-1}{2}\omega(H)+\lfloor \omega(H)/2 \rfloor\Big)}{k-1} \Bigg\rceil 
\\
\\
& = & \omega(H)+\Big\lceil \frac{2\lfloor \omega(H)/2 \rfloor}{k-1} \Big\rceil 
\\
\\
& = & f_k(\omega(H)), 
\end{array} 
\end{displaymath} 
which is what we needed. This proves the first statement of the lemma. 

It remains to prove the second statement of the lemma. We fix an integer $n \geq 2$, and we construct $H_n^k$ as follows. Let $X_1,\dots,X_k$ be pairwise disjoint sets such that for all $i \in \{1,\dots,k\}$, 
\begin{itemize} 
\item if $i$ is odd, then $|X_i| = \lfloor n/2 \rfloor$, and 
\item if $i$ is even, then $|X_i| = \lceil n/2 \rceil$. 
\end{itemize} 
Since $n \geq 2$, sets $X_1,\dots,X_k$ are all nonempty. Now, let $H_n^k$ be the graph whose vertex set is $V(H_n^k) = X_1 \cup \dots \cup X_k$, and with adjacency as follows: 
\begin{itemize} 
\item $X_1,\dots,X_k$ are all cliques; 
\item for all $i \in \{1,\dots,k\}$, $X_i$ is complete to $X_{i-1} \cup X_{i+1}$ and anticomplete to $V(H_n^k) \setminus (X_{i-1} \cup X_i \cup X_{i+1})$. 
\end{itemize} 
Clearly, $H_n^k$ is a $k$-hyperhole, and $\omega(H_n^k) = \lfloor n/2 \rfloor+\lceil n/2 \rceil = n$. It remains to show that $\chi(H_n^k) = f_k(n)$. But by the first statement of the lemma, we have that $\chi(H_n^k) \leq f_k(n)$, and so in fact, it suffices to show that $\chi(H_n^k) \geq f_k(n)$. 

It is clear that $\chi(H_n^k) \geq \Big\lceil \frac{|V(H_n^k)|}{\alpha(H_n^k)} \Big\rceil$. Further, by construction, and by the fact that $k$ is odd, we have that 
\begin{itemize} 
\item $\alpha(H_n^k) = \lfloor k/2 \rfloor = \frac{k-1}{2}$, and 
\item $|V(H_n^k)| = \lceil k/2 \rceil \lfloor n/2 \rfloor+\lfloor k/2 \rfloor \lceil n/2 \rceil = \frac{k-1}{2}n+\lfloor n/2 \rfloor$. 
\end{itemize} 
Thus, 
\begin{displaymath} 
\begin{array}{ccc cccc} 
\chi(H_n^k) & \geq & \Big\lceil \frac{|V(H_n^k)|}{\alpha(H_n^k)} \Big\rceil & = & \Bigg\lceil \frac{2\Big(\frac{k-1}{2}n+\lfloor n/2 \rfloor\Big)}{k-1} \Bigg\rceil & = & n+\Big\lceil \frac{2 \lfloor n/2 \rfloor}{k-1} \Big\rceil. 
\end{array} 
\end{displaymath} 
Lemma~\ref{lemma-fk-calculation} now implies that 
\begin{displaymath}
\begin{array}{ccc cc} 
\chi(H_n^k) & \geq & n+\Big\lceil \frac{2 \lfloor n/2 \rfloor}{k-1} \Big\rceil & = & f_k(n), 
\end{array} 
\end{displaymath} 
which is what we needed. This proves the second statement of the lemma. 
\end{proof} 

\begin{theorem} \label{thm-HHk-chi-optimal} Let $k \geq 4$ be an integer. Then $\mathcal{H}_k$ is $\chi$-bounded. Furthermore, if $k$ is even, then the identity function $i_{\mathbb{N}}$ is the optimal $\chi$-bounding function for $\mathcal{H}_k$, and if $k$ is odd, then $f_k$ is the optimal $\chi$-bounding function for $\mathcal{H}_k$. 
\end{theorem} 
\begin{proof} 
Note that every induced subgraph of a $k$-hyperhole is either a $k$-hyperhole or a chordal graph.\footnote{This is easy to see by inspection, but it also follows from Lemma~\ref{lemma-ring-chordal}(c).} Since chordal graphs are perfect (by~\cite{Berge-German, D61}), it follows that all graphs in $\mathcal{H}_k$ are either $k$-hyperholes or perfect graphs. Furthermore, by construction, $\mathcal{H}_k$ contains all $k$-hyperholes. Thus, if $k$ is odd, then Lemma~\ref{lemma-hyp-chi} implies that $f_k$ is the optimal $\chi$-bounding function for $\mathcal{H}_k$.\footnote{We are also using the fact that $K_1 \in \mathcal{H}_k$, $\omega(K_1) = 1$, and $\chi(K_1) = 1 = f_k(1)$.} Suppose now that $k$ is even. By Lemma~\ref{lemma-even-ring-col-alg}, all even hyperholes are perfect, and we deduce that all graphs in $\mathcal{H}_k$ are perfect. Furthermore, $\mathcal{H}_k$ contains all complete graphs. So, $i_{\mathbb{N}}$ is the optimal $\chi$-bounding function for $\mathcal{H}_k$. 
\end{proof} 

\begin{corollary} \label{cor-HHk-chi-optimal} Let $k \geq 4$ be an integer. Then $\mathcal{H}_{\geq k}$ is $\chi$-bounded. Furthermore, if $k$ is even, then $f_{k+1}$ is the optimal $\chi$-bounding function for $\mathcal{H}_{\geq k}$, and if $k$ is odd, then $f_k$ is the optimal $\chi$-bounding function for $\mathcal{H}_{\geq k}$. 
\end{corollary} 
\begin{proof} 
This follows immediately from Lemma~\ref{lemma-fTfkgk}(c) and Theorem~\ref{thm-HHk-chi-optimal}. 
\end{proof} 

\begin{lemma} \label{lemma-ring-chi-optimal} Let $k \geq 5$ be an odd integer. Then all $k$-rings $R$ satisfy $\chi(R) \leq f_k(\omega(R))$. Furthermore, there exists a sequence $\{R_n^k\}_{n=2}^{\infty}$ of $k$-rings such that for all integers $n \geq 2$, we have that $\omega(R_n^k) = n$ and $\chi(R_n^k) = f_k(n)$. 
\end{lemma} 
\begin{proof} 
Since every $k$-hyperhole is a $k$-ring, the second statement of the lemma follows immediately from the second statement of Lemma~\ref{lemma-hyp-chi}. It remains to prove the first statement. Let $R$ be a $k$-ring. Then by Theorem~\ref{thm-ring-hyperhole}, there exists a $k$-hyperhole $H$ in $R$ such that $\chi(R) = \chi(H)$. By Lemma~\ref{lemma-hyp-chi}, we have that $\chi(H) \leq f_k(\omega(H))$. Clearly, $\omega(H) \leq \omega(R)$, and by Lemma~\ref{lemma-fTfkgk}(b), $f_k$ is a nondecreasing function. We now have that 
\begin{displaymath} 
\begin{array}{ccc cccc} 
\chi(R) & = & \chi(H) & \leq & f_k(\omega(H)) & \leq & f_k(\omega(R)), 
\end{array} 
\end{displaymath} 
which is what we needed. This completes the argument. 
\end{proof} 

\begin{theorem} \label{thm-RRk-chi-optimal} Let $k \geq 4$ be an integer. Then $\mathcal{R}_k$ is $\chi$-bounded. Furthermore, if $k$ is even, then the identity function $i_{\mathbb{N}}$ is the optimal $\chi$-bounding function for $\mathcal{R}_k$, and if $k$ is odd, then $f_k$ is the optimal $\chi$-bounding function for $\mathcal{R}_k$. 
\end{theorem} 
\begin{proof} 
Suppose first that $k$ is even. By Lemma~\ref{lemma-even-ring-col-alg}, every $k$-ring $R$ satisfies $\chi(R) = \omega(R)$. Lemma~\ref{lemma-simplicial-chi} and an easy induction now imply that $\mathcal{R}_k$ is $\chi$-bounded by $i_{\mathbb{N}}$, and it is obvious that this $\chi$-bounding function is optimal. 

Suppose now that $k$ is odd. By Lemma~\ref{lemma-RRk-hered}, all $k$-rings belong to $\mathcal{R}_k$. Thus, it suffices to show that $\mathcal{R}_k$ is $\chi$-bounded by $f_k$, for optimality will then follow immediately from Lemma~\ref{lemma-ring-chi-optimal}.\footnote{We are also using the fact that $K_1 \in \mathcal{R}_k$, $\omega(K_1) = 1$, and $\chi(K_1) = 1 = f_k(1)$.} 

So, fix $G \in \mathcal{R}_k$, and assume inductively that all graphs $G' \in \mathcal{R}_k$ with $|V(G')|<|V(G)|$ satisfy $\chi(G') \leq f_k(\omega(G'))$. We must show that $\chi(G) \leq f_k(\omega(G))$. If $G$ is a complete graph, then $\chi(G) = \omega(G) \leq f_k(\omega(G))$, and we are done. So assume that $G$ is not complete, and in particular, $G$ has at least two vertices. 

Suppose that $G$ has a simplicial vertex $v$. Then by Lemma~\ref{lemma-simplicial-chi}, $\chi(G) = \max\{\omega(G),\chi(G \setminus v)\}$. Clearly, $\omega(G) \leq f_k(\omega(G))$. On the other hand, using the induction hypothesis and the fact that $f_k$ is nondecreasing (by Lemma~\ref{lemma-fTfkgk}(b)), we get that $\chi(G \setminus v) \leq f_k(\omega(G \setminus v)) \leq f_k(\omega(G))$. It now follows that $\chi(G) = \max\{\omega(G),\chi(G \setminus v)\} \leq f_k(\omega(G))$, which is what we needed. 

Suppose now that $G$ does not contain a simplicial vertex. Then by the definition of $\mathcal{R}_k$, $G$ is a $k$-ring, and so Lemma~\ref{lemma-ring-chi-optimal} implies that $\chi(G) \leq f_k(\omega(G))$. This completes the argument. 
\end{proof} 

\begin{corollary} \label{cor-RRk-chi-optimal} Let $k \geq 4$ be an integer. Then $\mathcal{R}_{\geq k}$ is $\chi$-bounded. Furthermore, if $k$ is even, then $f_{k+1}$ is the optimal $\chi$-bounding function for $\mathcal{R}_{\geq k}$, and if $k$ is odd, then $f_k$ is the optimal $\chi$-bounding function for $\mathcal{R}_{\geq k}$. 
\end{corollary} 
\begin{proof} 
This follows immediately from Lemma~\ref{lemma-fTfkgk}(c) and Theorem~\ref{thm-RRk-chi-optimal}. 
\end{proof} 

A {\em cobipartite graph} is a graph whose complement is bipartite. Equivalently, a graph is {\em cobipartite} if its vertex set can be partitioned into two (possibly empty) cliques. 

\begin{lemma} \label{lemma-hypanti-perfect} Let $k \geq 4$ be an integer, let $A$ be a $k$-hyperantihole, and let $(X_1,\dots,X_k)$ be a partition of $V(A)$ into nonempty cliques such that for all $i \in \{1,\dots,k\}$, $X_i$ is complete to $V(A) \setminus (X_{i-1} \cup X_i \cup X_{i+1})$ and anticomplete to $X_{i-1} \cup X_{i+1}$. Then for all $i \in \{1,\dots,k\}$, $A \setminus X_i$ is perfect. Furthermore, if $k$ is even, then $A$ is perfect. 
\end{lemma} 
\begin{proof} 
The Perfect Graph Theorem~\cite{Lovasz} states that a graph is perfect if and only if its complement is perfect; bipartite graphs are obviously perfect, and it follows that cobipartite graphs are also perfect. Clearly, for all $i \in \{1,\dots,k\}$, $A \setminus X_i$ is cobipartite and consequently perfect. Furthermore, if $k$ is even, then $A$ is cobipartite and consequently perfect. 
\end{proof} 

\begin{lemma} \label{lemma-hypanti-chi} Let $k \geq 5$ be an odd integer. Then all $k$-hyperantiholes $A$ satisfy $\omega(A) \geq \frac{k-1}{2}$ and $\chi(A) \leq g_k(\omega(A))$. Furthermore, there exists a sequence $\{A_n^k\}_{n=\frac{k-1}{2}}^{\infty}$ of $k$-hyperantiholes such that for all integers $n \geq \frac{k-1}{2}$, we have that $\omega(A_n^k) = n$ and $\chi(A_n^k) = g_k(n)$.\end{lemma} 
\begin{proof} 
We begin by proving the first statement of the lemma. Let $A$ be a $k$-hyperantihole, and let $(X_1,\dots,X_k)$ be a partition of $V(A)$ into nonempty cliques such that for all $i \in \{1,\dots,k\}$, $X_i$ is complete to $V(A) \setminus (X_{i-1} \cup X_i \cup X_{i+1})$ and anticomplete to $X_{i-1} \cup X_{i+1}$, as in the definition of a $k$-hyperantihole. Since $A$ is a $k$-hyperantihole, and since $k$ is odd, we see that $\omega(A) \geq \lfloor \frac{k}{2} \rfloor = \frac{k-1}{2}$. 

By symmetry, we may assume that $|X_2| = \min\{|X_1|,\dots,|X_k|\}$. Since $\bigcup\limits_{i=1}^{(k-1)/2} X_{2i}$ is a clique, we see that $\sum\limits_{i=1}^{(k-1)/2} |X_{2i}| \leq \omega(A)$, and so by the minimality of $|X_2|$, we have that $|X_2| \leq \Big\lfloor \frac{2\omega(A)}{k-1} \Big\rfloor$. 

By construction, $X_2$ is anticomplete to $X_1 \cup X_3$ in $A$, and $|X_2| \leq |X_1|,|X_3|$. Fix any $X_1^2 \subseteq X_1$ and $X_3^2 \subseteq X_3$ such that either $|X_1^2| = \Big\lfloor |X_2|/2 \Big\rfloor$ and $|X_3^2| = \Big\lceil |X_2|/2 \Big\rceil$, or $|X_1^2| = \Big\lceil |X_2|/2 \Big\rceil$ and $|X_3^2| = \Big\lfloor |X_2|/2 \Big\rfloor$.\footnote{This way, we maintain full symmetry between $X_1$ and $X_3$.} Let $X_2^* = X_1^2 \cup X_2 \cup X_3^2$. Note that $X_2$ and $X_2^* \setminus X_2 = X_1^2 \cup X_3^2$ are cliques in $A$, they are anticomplete to each other in $A$, and they are both of size $|X_2|$. Thus, $\chi(A[X_2^*]) = |X_2|$. 

By Lemma~\ref{lemma-hypanti-perfect}, $A \setminus X_2$ is perfect. Since $A \setminus X_2^*$ is an induced subgraph of $A \setminus X_2$, it follows that $\chi(A \setminus X_2^*) = \omega(A \setminus X_2^*)$. Let $K$ be a maximum clique of $A \setminus X_2^*$. (In particular, $K \cap X_2 = \emptyset$.) Then 
\begin{displaymath} 
\begin{array}{rcl} 
\chi(A) & \leq & \chi(A \setminus X_2^*)+\chi(A[X_2^*]) 
\\
\\
& = & \omega(A \setminus X_2^*)+|X_2| 
\\
\\
& = & |K|+|X_2| 
\\
\\
& = & |K \cup X_2|. 
\end{array} 
\end{displaymath} 

Suppose first that $K$ intersects neither $X_1 \setminus X_1^2$ nor $X_3 \setminus X_3^2$. Since $K \subseteq V(A) \setminus X_2^*$, it follows that $K \cap (X_1 \cup X_3) = \emptyset$. Then $X_2$ is complete to $K$. Thus, $K \cup X_2$ is a clique of $A$, and it follows that $|K \cup X_2| \leq \omega(A)$; consequently, 
\begin{displaymath} 
\begin{array}{ccc cccc} 
\chi(A) & \leq & |K \cup X_2| & \leq & \omega(A) & \leq & g_k(\omega(A)), 
\end{array} 
\end{displaymath} 
and we are done. 

Suppose now that $K$ intersects at least one of $X_1 \setminus X_1^2$ and $X_3 \setminus X_3^2$; by symmetry, we may assume that $K \cap (X_1 \setminus X_1^2) \neq \emptyset$. Then $K \cup X_1^2$ is a clique of $A$,\footnote{Since $K \subseteq V(A) \setminus X_2^*$ and $X_1^2 \subseteq X_2^*$, we have that $K$ and $X_1^2$ are disjoint.} and it follows that $|K \cup X_1^2| \leq \omega(A)$; consequently, 
\begin{displaymath} 
\begin{array}{ccc cc} 
|K| & \leq & \omega(A)-|X_1^2| & \leq & \omega(A)-\Big\lfloor |X_2|/2 \Big\rfloor, 
\end{array} 
\end{displaymath} 
and so 
\begin{displaymath} 
\begin{array}{rcl} 
\chi(A) & \leq & |K|+|X_2| 
\\
\\
& \leq & (\omega(A)-\Big\lfloor |X_2|/2 \Big\rfloor)+|X_2| 
\\
\\
& = & \omega(A)+\Big\lceil |X_2|/2 \Big\rceil 
\\
\\
& \leq & \omega(A)+\Big\lceil \Big\lfloor \frac{2\omega(A)}{k-1} \Big\rfloor/2 \Big\rceil. 
\end{array} 
\end{displaymath} 
By Lemma~\ref{lemma-gk-calculation}, we now have that 
\begin{displaymath} 
\begin{array}{ccc cc} 
\chi(A) & \leq & \omega(A)+\Big\lceil \Big\lfloor \frac{2\omega(A)}{k-1} \Big\rfloor/2 \Big\rceil & = & g_k(\omega(A)), 
\end{array} 
\end{displaymath} 
and again we are done. This proves the first statement of the lemma. 

It remains to prove the second statement of the lemma. We fix an integer $n \geq \frac{k-1}{2}$, and we construct $A_n^k$ as follows. Set $m = \lfloor \frac{n}{k-1} \rfloor$ and $\ell = n-(k-1)m$. Clearly, $m$ is a nonnegative integer, $\ell \in \{0,\dots,k-2\}$, $n = (k-1)m+\ell$, and $n \equiv \ell$ (mod $k-1$). Now, let $X_1,\dots,X_k$ be pairwise disjoint sets such that for all $i \in \{1,\dots,k\}$, 
\begin{itemize} 
\item if $0 \leq \ell \leq \frac{k-3}{2}$, then $|X_1| = \dots = |X_{2\ell}| = 2m+1$ and $|X_{2\ell+1}| = \dots = |X_k| = 2m$; 
\item if $\frac{k-1}{2} \leq \ell \leq k-2$, then $|X_1| = \dots = |X_{2\ell-k+1}| = 2m+2$ and $|X_{2\ell-k+2}| = \dots = |X_k| = 2m+1$. 
\end{itemize} 
Since $n \geq \frac{k-1}{2}$, sets $X_1,\dots,X_k$ are all nonempty. Let $A_n^k$ be the graph with vertex set $V(A_n^k) = X_1 \cup \dots \cup X_k$, and with adjacency as follows: 
\begin{itemize} 
\item $X_1,\dots,X_k$ are all cliques; 
\item for all $i \in \{1,\dots,k\}$, $X_i$ is complete to $V(A_n^k) \setminus (X_{i-1} \cup X_i \cup X_{i+1})$ and anticomplete to $X_{i-1} \cup X_{i+1}$. 
\end{itemize} 
Clearly, $A_n^k$ is a $k$-hyperantihole. We must show that $\omega(A_n^k) = n$ and $\chi(A_n^k) = g_k(n)$. 

We first show that $\omega(A_n^k) = n$. Suppose first that $0 \leq \ell \leq \frac{k-3}{2}$. Now $2\ell$ consecutive $X_i$'s are of size $2m+1$ (since they are consecutive, at most $\ell$ of them can be included in a clique of $A_n^k$), and all the other $X_i$'s are of size $2m$. So, a maximum clique of $A_n^k$ is the union of $\ell$ sets $X_i$ of size $2m+1$, and of $\frac{k-1}{2}-\ell$ sets $X_i$ of size $2m$. It follows that 
\begin{displaymath} 
\begin{array}{ccc ccc ccc} 
\omega(A_n^k) & = & \ell(2m+1)+\Big(\frac{k-1}{2}-\ell\Big)2m & = & (k-1)m+\ell & = & n, 
\end{array} 
\end{displaymath} 
which is what we needed. 

Suppose now that $\frac{k-1}{2} \leq \ell \leq k-2$. Then $2\ell-k+1$ consecutive $X_i$'s are of size $2m+2$ (since they are consecutive, at most $\lceil \frac{2\ell-k+1}{2} \rceil = \ell-\frac{k-1}{2}$ of them can be included in a clique of $A_n^k$), and all the other $X_i$'s are of size $2m+1$. So, a maximum clique of $A_n^k$ is the union of $\ell-\frac{k-1}{2}$ sets $X_i$ of size $2m+2$, and of $\frac{k-1}{2}-(\ell-\frac{k-1}{2}) = k-\ell-1$ sets $X_i$ of size $2m+1$. It follows that 
\begin{displaymath} 
\begin{array}{rcl} 
\omega(A_n^k) & = & \Big(\ell-\frac{k-1}{2}\Big)(2m+2)+(k-\ell-1)(2m+1) 
\\
\\
& = & (k-1)m+\ell 
\\
\\
& = & n, 
\end{array} 
\end{displaymath} 
which is what we needed. 

We have now shown that $\omega(A_n^k) = n$. It remains to show that $\chi(A_n^k) = g_k(n)$. But by the first statement of the lemma, we have that $\chi(A_n^k) \leq g_k(n)$, and so in fact, it suffices to show that $\chi(A_n^k) \geq g_k(n)$. Clearly, $\chi(A_n^k) \geq \Big\lceil \frac{|V(A_n^k)|}{\alpha(A_n^k)} \Big\rceil$, and since $A_n^k$ is a hyperantihole, we see that $\alpha(A_n^k) = 2$. Thus, $\chi(A_n^k) \geq \Big\lceil \frac{1}{2}|V(A_n^k)| \Big\rceil$. 

Suppose first that $0 \leq \ell \leq \frac{k-3}{2}$. Then $g_k(n) = \lfloor \frac{kn}{k-1} \rfloor$, and we have that 
\begin{displaymath} 
\begin{array}{rcl} 
\chi(A_k^n) & \geq & \Big\lceil \frac{1}{2}|V(A_n^k)| \Big\rceil
\\
\\
& = & \Big\lceil \frac{1}{2}\Big(2\ell(2m+1)+(k-2\ell)2m\Big) \Big\rceil
\\
\\
& = & km+\ell 
\\
\\
& = & n+m 
\\
\\
& = & \lfloor \frac{kn}{k-1} \rfloor 
\\
\\
& = & g_k(n), 
\end{array} 
\end{displaymath} 
which is what we needed. 

Suppose now that $\frac{k-1}{2} \leq \ell \leq k-2$. Since $\ell \neq 0$, we see that $\frac{kn}{k-1}$ is not an integer, and so $\lfloor \frac{kn}{k-1} \rfloor+1 = \lceil \frac{kn}{k-1} \rceil$. Further, since $\frac{k-1}{2} \leq \ell \leq k-2$, we have that $g_k(n) = \lceil \frac{kn}{k-1} \rceil$. We then see that 
\begin{displaymath} 
\begin{array}{rcl} 
\chi(A_k^n) & \geq & \Big\lceil \frac{1}{2}|V(A_n^k)| \Big\rceil
\\
\\
& = & \Big\lceil \frac{1}{2}\Big((2\ell-k+1)(2m+2)+(2k-2\ell-1)(2m+1)\Big) \Big\rceil
\\
\\
& = & \Big\lceil km+\ell+\frac{1}{2} \Big\rceil
\\
\\
& = & km+\ell+1 
\\
\\
& = & n+m+1 
\\
\\
& = & \lfloor \frac{kn}{k-1} \rfloor+1  
\\
\\
& = & \lceil \frac{kn}{k-1} \rceil 
\\
\\
& = & g_k(n), 
\end{array} 
\end{displaymath} 
which is what we needed. This proves the second statement of the lemma. 
\end{proof} 

\begin{theorem} \label{thm-AAk-chi-optimal} Let $k \geq 4$ be an integer. Then $\mathcal{A}_k$ is $\chi$-bounded. Furthermore, if $k$ is even, then the identity function $i_{\mathbb{N}}$ is the optimal $\chi$-bounding function for $\mathcal{A}_k$, and if $k$ is odd, then $g_k$ is the optimal $\chi$-bounding function for $\mathcal{A}_k$. 
\end{theorem} 
\begin{proof} 
If $k$ is even, then by Lemma~\ref{lemma-hypanti-perfect}, all graphs in $\mathcal{A}_k$ are perfect, and it follows that $i_{\mathbb{N}}$ is the optimal $\chi$-bounding function for $\mathcal{A}_k$. 

Suppose now that $k$ is odd. Clearly, all $k$-hyperantiholes belong to $\mathcal{A}_k$; on the other hand, it follows from Lemma~\ref{lemma-hypanti-perfect} that all graphs in $\mathcal{A}_k$ are either $k$-hyperantiholes or perfect graphs. So, by Lemma~\ref{lemma-hypanti-chi}, $\mathcal{A}_k$ is $\chi$-bounded by $g_k$. It remains to establish the optimality of $g_k$. Fix $n \in \mathbb{N}$. If $n \leq \frac{k-3}{2}$, then $g_k(n) = n$, and we observe that $K_n \in \mathcal{A}_k$, $\omega(K_n) = n$, and $\chi(K_n) = n = g_k(n)$. On the other hand, if $n \geq \frac{k-1}{2}$, then we let $A_n^k$ be as in Lemma~\ref{lemma-hypanti-chi}, and we observe that $A_n^k \in \mathcal{A}_k$, $\omega(A_n^k) = n$, and $\chi(A_n^k) = g_k(n)$. This proves that the $\chi$-bounding function $g_k$ for $\mathcal{A}_k$ is indeed optimal. 
\end{proof} 

\begin{corollary} \label{cor-AAk-chi-optimal} Let $k \geq 4$ be an integer. Then $\mathcal{A}_{\geq k}$ is $\chi$-bounded. Furthermore, if $k$ is even, then $g_{k+1}$ is the optimal $\chi$-bounding function for $\mathcal{A}_{\geq k}$, and if $k$ is odd, then $g_k$ is the optimal $\chi$-bounding function for $\mathcal{A}_{\geq k}$. 
\end{corollary} 
\begin{proof} 
This follows immediately from Lemma~\ref{lemma-fTfkgk}(c) and Theorem~\ref{thm-AAk-chi-optimal}. 
\end{proof} 

We remind the reader that the function $f_{\text{T}}:\mathbb{N} \rightarrow \mathbb{N}$ is given by 
\begin{displaymath} 
\begin{array}{lll} 
f_{\text{T}}(n) & = & \left\{\begin{array}{lll} 
\lfloor 5n/4 \rfloor & \text{if} & \text{$n \equiv 0,1$ (mod $4$)} 
\\
\\
\lceil 5n/4 \rceil & \text{if} & \text{$n \equiv 2,3$ (mod $4$)} 
\end{array}\right. 
\end{array}
\end{displaymath} 
for all $n \in \mathbb{N}$.

Note that $\mathcal{H}_{\geq 4}$ is the class of all induced subgraphs of hyperholes, and that $\mathcal{A}_{\geq 4}$ is the class of all induced subgraphs of hyperantiholes. Furthermore, by Lemma~\ref{lemma-RRk-hered}, $\mathcal{R}_{\geq 4}$ contains all induced subgraphs of rings. In particular, $\mathcal{H}_{\geq 4} \subseteq \mathcal{R}_{\geq 4}$. 

\begin{theorem} \label{thm-HH-AA-RR-chi-optimal} Classes $\mathcal{H}_{\geq 4}$, $\mathcal{A}_{\geq 4}$, and $\mathcal{R}_{\geq 4}$ are $\chi$-bounded. Furthermore, $f_{\text{T}}$ is the optimal $\chi$-bounding function for all three classes. 
\end{theorem} 
\begin{proof} 
By Lemma~\ref{lemma-fTfkgk}, we have that $f_{\text{T}} = f_5 = g_5$. The result now follows immediately from Corollaries~\ref{cor-HHk-chi-optimal},~\ref{cor-RRk-chi-optimal}, and~\ref{cor-AAk-chi-optimal}. 
\end{proof} 

\begin{theorem} \label{thm-GT-chi} $\mathcal{G}_{\text{T}}$ is $\chi$-bounded. Furthermore, $f_{\text{T}}$ is the optimal $\chi$-bounding function for $\mathcal{G}_{\text{T}}$. 
\end{theorem} 
\begin{proof} 
We begin by showing $f_{\text{T}}$ is a $\chi$-bounding function for $\mathcal{G}_{\text{T}}$. First, by Lemma~\ref{lemma-fTfkgk}, we have that $f_{\text{T}}$ is nondecreasing, and that $f_{\text{T}} = f_5 = g_5$. Now, fix $G \in \mathcal{G}_{\text{T}}$, and assume inductively that for all $G' \in \mathcal{G}_{\text{T}}$ such that $|V(G')| < |V(G)|$, we have that $\chi(G') \leq f_{\text{T}}(\omega(G'))$. 

By Theorem~\ref{thm-GT-decomp}, we know that either $G$ is a complete graph, a ring, or a 7-hyperantihole, or $G$ admits a clique-cutset. If $G$ is a complete graph, a ring, or a 7-hyperantihole, then $G \in \mathcal{R}_{\geq 4} \cup \mathcal{A}_{\geq 4}$, and Theorem~\ref{thm-HH-AA-RR-chi-optimal} guarantees that $\chi(G) \leq f_{\text{T}}(\omega(G))$. It remains to consider the case when $G$ admits a clique-cutset. Let $(A,B,C)$ be a clique-cut-partition of $G$, and set $G_A = G[A \cup C]$ and $G_B = G[B \cup C]$. Clearly, $\chi(G) = \max\{\chi(G_A),\chi(G_B)\}$. Using the induction hypothesis and the fact that $f_{\text{T}}$ is nondecreasing, we now get that 
\begin{displaymath} 
\begin{array}{rcl} 
\chi(G) & = & \max\{\chi(G_A),\chi(G_B)\} 
\\
\\
& \leq & \max\{f_{\text{T}}(\omega(G_A)),f_{\text{T}}(\omega(G_B))\} 
\\
\\
& \leq & f_{\text{T}}(\omega(G)), 
\end{array} 
\end{displaymath} 
which is what we needed. This proves that $f_{\text{T}}$ is indeed a $\chi$-bounding function for $\mathcal{G}_{\text{T}}$. 

It remains to establish the optimality of $f_{\text{T}}$. Let $n \in \mathbb{N}$; we must exhibit a graph $G \in \mathcal{G}_{\text{T}}$ such that $\omega(G) = n$ and $\chi(G) = f_{\text{T}}(n)$. If $n = 1$, then we observe that $K_1 \in \mathcal{G}_{\text{T}}$, $\omega(K_1) = 1$, and $\chi(K_1) = 1 = f_{\text{T}}(1)$. So assume that $n \geq 2$. Let $H_n^5$ be as in the statement of Lemma~\ref{lemma-hyp-chi}. Then $H_n^5$ is a 5-hyperhole, and it is easy to see that all hyperholes belong to $\mathcal{G}_{\text{T}}$;\footnote{Alternatively, we observe that every hyperhole is a ring, and by Lemma~\ref{lemma-ring-chordal}(d), all rings belong to $\mathcal{G}_{\text{T}}$.} thus, $H_n^5 \in \mathcal{G}_{\text{T}}$. Further, since $f_{\text{T}} = f_5$, Lemma~\ref{lemma-hyp-chi} guarantees that $\omega(H_n^5) = n$ and $\chi(H_n^5) = f_5(n) = f_{\text{T}}(n)$. Thus, $f_{\text{T}}$ is indeed the optimal $\chi$-bounding function for $\mathcal{G}_{\text{T}}$. 
\end{proof}

\section{Class $\mathcal{G}_{\text{T}}$ and Hadwiger's conjecture} \label{sec:Hadwiger} 

In this section, we prove Hadwiger's conjecture for the class $\mathcal{G}_{\text{T}}$ (see Theorem~\ref{thm-Hadwiger-GT}). Recall that a graph is {\em perfect} if all its induced subgraphs $H$ satisfy $\chi(H) = \omega(H)$. Obviously, Hadwiger's conjecture is true for perfect graphs: every perfect graph $G$ contains $K_{\chi(G)}$ an induced subgraph, and therefore as a minor as well. 

\begin{lemma} \label{lemma-Hadwiger-hyperhole} Every hyperhole $H$ contains $K_{\chi(H)}$ as a minor. 
\end{lemma} 
\begin{proof} 
Let $H$ be a hyperhole, and let $k$ be its length. Let $(X_1,\dots,X_k)$ be a partition of $V(H)$ into nonempty cliques such that for all $i \in \{1,\dots,k\}$, $X_i$ is complete to $X_{i-1} \cup X_{i+1}$ and anticomplete to $V(H) \setminus (X_{i-1} \cup X_i \cup X_{i+1})$. By symmetry, we may assume that $|X_1| = \min\{|X_1|,|X_2|,\dots,|X_k|\}$. Clearly, $\chi(H \setminus X_1) = \omega(H \setminus X_1)$,\footnote{By Lemma~\ref{lemma-ring-chordal}(c), $H \setminus X_1$ is chordal, and by~\cite{Berge-German, D61}, chordal graphs are perfect. So, $H \setminus X_1$ is perfect and therefore satisfies $\chi(H \setminus X_1) = \omega(H \setminus X_1)$.} and furthermore, there exists some index $j \in \{2,\dots,k-1\}$ such that $\omega(H \setminus X_1) = |X_j \cup X_{j+1}|$. By the choice of $X_1$, we see that there are $|X_1|$ vertex-disjoint induced paths between $X_{j-1}$ and $X_{j+2}$, none of them passing through $X_j \cup X_{j+1}$. We then take our $|X_1|$ paths and the vertices of $X_j \cup X_{j+1}$ as branch sets, and we obtain a $K_{|X_1|+\omega(H \setminus X_1)}$ minor in $G$. Since $\chi(H) \leq |X_1|+\chi(H \setminus X_1) = |X_1|+\omega(H \setminus X_1)$, we conclude that $H$ contains $K_{\chi(H)}$ as a minor. 
\end{proof} 

\begin{lemma} \label{lemma-Hadwiger-ring} Every ring $R$ contains $K_{\chi(R)}$ as a minor. 
\end{lemma} 
\begin{proof} 
This follows immediately from Theorem~\ref{thm-ring-hyperhole} and Lemma~\ref{lemma-Hadwiger-hyperhole}.
\end{proof}

\begin{lemma} \label{lemma-Hadwiger-hyperantihole} Every hyperantihole $A$ contains $K_{\chi(A)}$ as a minor. 
\end{lemma} 
\begin{proof} 
Let $A$ be a hyperantihole, and let $(X_1,\dots,X_k)$, with $k \geq 4$, be a partition of $V(A)$ into nonempty cliques, such that for all $i \in \{1,\dots,k\}$, $X_i$ is complete to $A \setminus (X_{i-1} \cup X_i \cup X_{i+1})$ and anticomplete to $X_{i-1} \cup X_{i+1}$, as in the definition of a hyperantihole. If $k = 4$, then $V(K)$ can be partitioned into two cliques (namely $X_1 \cup X_3$ and $X_2 \cup X_4$), anticomplete to each other, and the result is immediate. From now on, we assume that $k \geq 5$. 

By symmetry, we may assume that $|X_1| = \min\{|X_1|,|X_2|,\dots,|X_k|\}$. Clearly, $\chi(A) \leq \chi(A \setminus X_1)+|X_1|$.  On the other hand, by Lemma~\ref{lemma-hypanti-perfect}, $A \setminus X_1$ is perfect, and in particular, $\chi(A \setminus X_1) = \omega(A \setminus X_1)$. Let $K$ be a clique of size $\omega(A \setminus X_1)$ in $A \setminus X_1$. Then, $\chi(A) \leq |K|+|X_1|$, and so it suffices to show that $A$ contains $K_{|K|+|X_1|}$ as a minor. 

If $K \cap (X_k \cup X_2) = \emptyset$, then $X_1$ is complete to $K$ in $A$, $K \cup X_1$ is a clique of size $|K|+|X_1|$ in $A$, and we are done. 

From now on, we assume that $K$ intersects at least one of $X_2$ and $X_k$. By symmetry, we may assume that $K \cap X_2 \neq \emptyset$. Since $X_2$ is anticomplete to $X_3$, and since $K$ is a clique, we see that $K \cap X_3 = \emptyset$. Since $X_{k-1}$ and $X_k$ are anticomplete to each other, and since $K$ is a clique, we see that $K$ intersects at most one of $X_{k-1},X_k$, and we deduce that $|(X_{k-1} \cup X_k) \setminus K| \geq \min\{|X_{k-1}|,|X_k|\} \geq |X_1|$. So, there exist $|X_1|$ pairwise disjoint three-vertex subsets of $V(A) \setminus K$, each of them containing exactly one vertex from each of the sets $X_1$, $X_3$, and $X_{k-1} \cup X_k$. Clearly, each of these three-vertex sets induces a connected subgraph of $A$. We now take our $|X_1|$ three-vertex sets and all the vertices of $K$ as branch sets, and we obtain a $K_{|K|+|X_1|}$ minor in $A$. This completes the argument. 
\end{proof} 

\begin{theorem} \label{thm-Hadwiger-GT} Every graph $G \in \mathcal{G}_{\text{T}}$ contains $K_{\chi(G)}$ as a minor. 
\end{theorem} 
\begin{proof}
Fix $G \in \mathcal{G}_{\text{T}}$, and assume inductively that every graph $G' \in \mathcal{G}_{\text{T}}$ with $|V(G')| < |V(G)|$ contains $K_{\chi(G')}$ as a minor. We must show that $G$ contains $K_{\chi(G)}$ as a minor. We apply Theorem~\ref{thm-GT-decomp}. Suppose first that $G$ admits a clique-cutset, and let $(A,B,C)$ be a clique-cut-partition of $G$. Clearly, $\chi(G) = \max\{\chi(G[A \cup C]),\chi(G[B \cup C])\}$, and the result follows from the induction hypothesis. So assume that $G$ does not admit a clique-cutset. Then Theorem~\ref{thm-GT-decomp} implies that $G$ is a complete graph, a ring, or a 7-hyperantihole; in the first case, the result is immediate, in the second, it follows from Lemma~\ref{lemma-Hadwiger-ring}, and in the third, it follows from Lemma~\ref{lemma-Hadwiger-hyperantihole}. 
\end{proof}

\small{

} 

\end{document}